\documentclass[10pt]{amsart}
\usepackage{amsmath}
\usepackage{amscd,amsthm,amssymb,amsfonts}
\usepackage{mathrsfs}
\usepackage{dsfont}
\usepackage{stmaryrd}
\usepackage{euscript}
\usepackage{expdlist}
\usepackage{enumerate}

\input xy
\xyoption{all}
\usepackage[OT2,T1]{fontenc}

\theoremstyle{plain}
\newtheorem{thm}{Theorem}[section]

\newtheorem*{thma}{Theorem A}

\newtheorem*{thm*}{Theorem}

\newtheorem{lm}[thm]{Lemma}
\newtheorem{cor}[thm]{Corollary}
\newtheorem*{cor*}{Corollary}
\newtheorem{prop}[thm]{Proposition}
\newtheorem*{conj*}{Conjecture}
\newtheorem{conj}{Conjecture}

\theoremstyle{remark}
\newtheorem*{remark}{Remark}

\theoremstyle{definition}
\newtheorem*{defn*}{Definition}
\newtheorem{Remark}[thm]{Remark}
\newtheorem{I_Remark*}{Remark}
\newtheorem{defn}[thm]{Definition}

\newcommand{\nc}{\newcommand}

\newcommand{\beq}{\begin{equation}}
\newcommand{\eeq}{\end{equation}}
\newcommand{\bpmx}{\begin{pmatrix}}
\newcommand{\epmx}{\end{pmatrix}}
\newcommand{\bbmx}{\begin{bmatrix}}
\newcommand{\ebmx}{\end{bmatrix}}
\newcommand{\wh}{\widehat}
\newcommand{\wtd}{\widetilde}

\newcommand{\beqcd}[1]{\begin{equation*}\label{#1}\tag{#1}}
\newcommand{\eeqcd}{\end{equation*}}

\numberwithin{equation}{section}

\def\parref#1{\ref{#1}}
\def\thmref#1{Theorem~\parref{#1}}

\def\propref#1{Proposition~\parref{#1}}
\def\corref#1{Corollary~\parref{#1}}     \def\remref#1{Remark~\parref{#1}}
\def\secref#1{\S\parref{#1}}

\def\lmref#1{Lemma~\parref{#1}}
\def\subsecref#1{\S\parref{#1}}

\def\conjref#1{Conjecture~\parref{#1}}
\def\makeop#1{\expandafter\def\csname#1\endcsname
  {\mathop{\rm #1}\nolimits}\ignorespaces}

\makeop{Hom}   \makeop{End}   \makeop{Aut}   
\makeop{Pic} \makeop{Gal}       \makeop{Div} \makeop{Lie}
\makeop{PGL}   \makeop{Corr} \makeop{PSL} \makeop{sgn} \makeop{Spf}
 \makeop{Tr} \makeop{Nr} \makeop{Fr} \makeop{disc}
\makeop{Proj} \makeop{supp} \makeop{ker}   \makeop{Im} \makeop{dom}
\makeop{coker} \makeop{Stab} \makeop{SO} \makeop{SL} \makeop{SL}
\makeop{Cl}    \makeop{cond} \makeop{Br} \makeop{inv} \makeop{rank}
\makeop{id}    \makeop{Fil} \makeop{Frac}  \makeop{GL} \makeop{SU}
\makeop{Trd}   \makeop{Sp} \makeop{Tr}    \makeop{Trd} \makeop{Res}
\makeop{ind} \makeop{depth} \makeop{Tr} \makeop{st} \makeop{Ad}
\makeop{Int} \makeop{tr}    \makeop{Sym} \makeop{can} \makeop{SO}
\makeop{torsion} \makeop{GSp} \makeop{Tor}\makeop{Ker} \makeop{rec}
\makeop{Ind} \makeop{Coker}
 \makeop{vol} \makeop{Ext} \makeop{gr} \makeop{ad}
 \makeop{Gr}\makeop{corank} \makeop{Ann}
\makeop{Hol} 
\makeop{Fitt} \makeop{Mp} \makeop{CAP}

\def\Ord{{\mathrm{ord}}}

\def\Spec{\mathrm{Spec}\,}

\DeclareMathAlphabet{\mathpzc}{OT1}{pzc}{m}{it}
\DeclareSymbolFont{cyrletters}{OT2}{wncyr}{m}{n}
\DeclareMathSymbol{\SHA}{\mathalpha}{cyrletters}{"58}

\def\makebb#1{\expandafter\def
  \csname bb#1\endcsname{{\mathbb{#1}}}\ignorespaces}
\def\makebf#1{\expandafter\def\csname bf#1\endcsname{{\bf
      #1}}\ignorespaces}
\def\makegr#1{\expandafter\def
  \csname gr#1\endcsname{{\mathfrak{#1}}}\ignorespaces}
\def\makescr#1{\expandafter\def
  \csname scr#1\endcsname{{\EuScript{#1}}}\ignorespaces}
\def\makecal#1{\expandafter\def\csname cal#1\endcsname{{\mathcal
      #1}}\ignorespaces}

\def\doLetters#1{#1A #1B #1C #1D #1E #1F #1G #1H #1I #1J #1K #1L #1M
                 #1N #1O #1P #1Q #1R #1S #1T #1U #1V #1W #1X #1Y #1Z}
\def\doletters#1{#1a #1b #1c #1d #1e #1f #1g #1h #1i #1j #1k #1l #1m
                 #1n #1o #1p #1q #1r #1s #1t #1u #1v #1w #1x #1y #1z}
\doLetters\makebb   \doLetters\makecal  \doLetters\makebf
\doLetters\makescr
\doletters\makebf   \doLetters\makegr   \doletters\makegr

    \def\setminus{\smallsetminus}

\normalsize

\makeop{Ram} \makeop{Rep} \makeop{mass}

\makeop{Bl}
\def\abs#1{\left|#1\right|}
\def\norm#1{\lVert#1\rVert}

\def\Fp{{\mathbb F}_p}

\def\Qbarp{\C_p}
\def\Qp{\Q_p}
\def\Qbar{\ol{\Q}}

\def\Zp{\Z_p}

\def\rmN{{\mathrm N}}

\def\cA{{\mathcal A}}  

\def\cD{\mathcal D}
\def\cE{{\mathcal E}}
\def\cF{{\mathcal F}}  
\def\cG{{\mathcal G}}
\def\cL{{\mathcal L}}

\def\cJ{\mathcal J}
\def\cK{{\mathcal K}}  

\def\cO{\mathcal O}
\def\cS{{\mathcal S}}
\def\cf{{\mathcal f}}
\def\cW{{\mathcal W}}

\def\cC{\mathcal C}
\def\cJ{\mathcal J}

\def\bfc{\mathbf c}

\def\bfK{\mathbf K}

\def\bfT{{\mathbf T}}

\def\bftheta{\boldsymbol{\theta}}

\def\sG{\mathscr G}
\def\sH{\mathscr H}

\def\sL{\mathscr L}

\def\sA{\mathscr A}
\def\sV{\mathscr V}
\def\sK{\mathscr K}

\def\sT{\mathscr T}

\def\bbI{\mathbb I}

\newcommand{\Z}{\mathbf Z}
\newcommand{\Q}{\mathbf Q}
\newcommand{\R}{\mathbf R}
\newcommand{\C}{\mathbf C}
\newcommand{\A}{\mathbf A}    

\def\fraka{{\mathfrak a}}

\def\frakc{{\mathfrak c}}

\def\frakf{\mathfrak f}

\def\frakp{{\mathfrak p}}
\def\frakP{\mathfak P}
\def\frakq{\mathfrak q}

\def\frakt{\mathfrak t}

\def\frakl{\mathfrak l}
\def\frakP{\mathfrak P}

\def\frakH{{\mathfrak H}}

\def\bfone{{\mathbf 1}}

\def\BS{Bruhat-Schwartz }

\def\Teich{Teichm\"{u}ller }

\def\Frob{\mathrm{Frob}}

\newcommand{\<}{\langle}   
\renewcommand{\>}{\rangle} 

\def\isoto{\stackrel{\sim}{\to}}

\def\ot{\otimes}

\def\hookto{\hookrightarrow}
\def\longto{\longrightarrow}
\def\ol{\overline}  \nc{\opp}{\mathrm{opp}} \nc{\ul}{\underline}

\newcommand{\pair}[2]{\< #1, #2\>}

\newcommand{\pairing}{\pair{\,}{\,}}

\def\XYmatrix{\xymatrix@M=8pt} 
\def\ncmd{\newcommand}
\ncmd{\xysubset}[1][r]{\ar@<-2.5pt>@{^(-}[#1]\ar@<2.5pt>@{_(-}[#1]}
\ncmd{\XYmatrixc}[1]{\vcenter{\XYmatrix{#1}}}
\ncmd{\xyto}[1][r]{\ar@{->}[#1]}
\ncmd{\xyinj}[1][r]{\ar@{^(->}[#1]}
\ncmd{\xysurj}[1][r]{\ar@{->>}[#1]}
\ncmd{\xyline}[1][r]{\ar@{-}[#1]}
\ncmd{\xydotsto}[1][r]{\ar@{.>}[#1]}
\ncmd{\xydots}[1][r]{\ar@{.}[#1]}
\ncmd{\xyleadsto}[1][r]{\ar@{~>}[#1]}
\ncmd{\xyeq}[1][r]{\ar@{=}[#1]} \ncmd{\xyequal}[1][r]{\ar@{=}[#1]}
\ncmd{\xyequals}[1][r]{\ar@{=}[#1]}
\ncmd{\xymapsto}[1][r]{l\ar@{|->}[#1]}\ncmd{\xyimplies}[1][r]{\ar@{=>}[#1]}
\ncmd{\xyiso}{\ar[r]_-{\sim}}
\def\injxy{\ar@{^(->}}

\newcommand{\pMX}[4]{\begin{pmatrix}
{#1}& {#2}\\
{#3}&{#4}\end{pmatrix} }

 \newcommand{\pDII}[2]{\begin{pmatrix}{#1}&0
 \\0&{#2}\end{pmatrix}}

\newcommand{\seesaw}[4]{{#1}\ar@{-}[rd]\ar@{-}[d]&{#2}\ar@{-}[d]\\
{#3}\ar@{-}[ru]&{#4}}

\def\ie{i.e. }

\def\cf{\mbox{{\it cf.} }}

\def\uf{\varpi} 
\def\Abs{{|\!\cdot\!|}} 

\def\ndivides{\nmid}

\def\x{{\times}}

\def\onehalf{{\frac{1}{2}}}
\def\e{\varepsilon} 
\def\al{\alpha}

\def\Lam{\Lambda}

\def\om{\omega}

\def\iso{\simeq}
\def\con{\equiv}
\def\bksl{\backslash}
\newcommand\stt[1]{\left\{#1\right\}}
\def\ep{\epsilon}

\def\lam{\lambda}

\def\sg{\sigma}

\def\disjoint{\sqcup}

\newcommand{\powerseries}[1]{\llbracket{#1}\rrbracket}

\renewcommand\pmod[1]{\,(\mbox{mod }{#1})}

\renewcommand\Re{\text{Re}\,}
\newcommand\Dmd[1]{\left<{#1}\right>} 

\def\Cp{\C_p}

\usepackage{xcolor}
\usepackage{hyperref}
\usepackage[T1]{fontenc}
\usepackage[english,francais]{babel}
 \title[The derivative formula of $p$-adic $L$-functions at trivial zeros]{The derivative formula of $p$-adic $L$-functions for imaginary quadratic fields at trivial zeros}
\author[M. Chida]{Masataka Chida}
\address{
Tokyo Denki University, Tokyo 120-8551, Japan.
}
\email{chida@mail.dendai.ac.jp}
\author[M.-L. Hsieh]{Ming-Lun Hsieh}
\address{Institute of Mathematics, Academia Sinica, Taipei 10617, Taiwan}
\email{mlhsieh@math.sinica.edu.tw}

\def\Mat{\mathrm{M}}
\def\rmH{{\rm H}}

\def\rmd{\mathrm{d}}

\def\pbar{{\ol{\frakp}}}
\def\loc{{\rm loc}}
\def\cyc{\varepsilon_{\rm cyc}}
\def\Cp{\ol{\Q}_p}
\def\Frob{{\rm Fr}}

\def\Om{\boldsymbol \om}
\def\rmt{{\rm t}}
\def\frake{\mathfrak e}
\def\brch{\phi}
\def\e{\boldsymbol\varepsilon}
\def\cyc{\e_{\rm cyc}}
 \def\aA{\cC(\brch,\brch^c)}
\def\bB{\cC(\brch^c,\brch)}
\def\ux{\mathfrak u^\chi}
\def\vx{\mathfrak u^\chi_{\pbar}}
\thanks{Hsieh was partially supported by a MOST grant MOST 108-2628-M-001-009-MY4 and 110-2628-M-001-004-. Chida was supported by JSPS KAKENHI Grant Number JP18K03202.}
\begin{document}

 \maketitle
 To Bernadette Perrin-Riou, on her 65th birthday
 \selectlanguage{english}
 \begin{abstract}The rank one Gross conjecture for Deligne-Ribet $p$-adic $L$-functions was solved in works of Darmon-Dasgupta-Pollack and Ventullo by the Eisenstein congruence among Hilbert modular forms. The purpose of this paper is to prove an analogue of the Gross conjecture for the Katz $p$-adic $L$-functions attached to imaginary quadratic fields via the congruences between CM forms and non-CM forms. The new ingredient is to apply the $p$-adic Rankin-Selberg method to construct a non-CM Hida family which is congruent to a Hida family of CM forms at the $1+\varepsilon$ specialization. 

\medskip

\selectlanguage{french}
\noindent\textsc{R\'esum\'e.}
Le conjecture de Gross en rang 1 pour les fonctions $L$ $p$-adiques de 
Deligne-Ribet a \'et\'e r\'esolue par Darmon-Dasgupta-Pollack et Ventullo au moyen de congruences d'Eisenstein parmi les formes 
modulaires de Hilbert. Le but de cet article est de prouver un analogue 
de la conjecture de Gross pour les fonctions $L$ $p$-adiques de Katz des 
corps quadratiques imaginaires, via les congruences entre formes CM et 
formes non-CM. Le nouvel ingr\'edient est l'application de la m\'ethode de 
Rankin-Selberg $p$-adique pour construire une famille de Hida non-CM qui 
est congruente \`a une famille de Hida de formes CM pour la sp\'ecialisation 
$1+\varepsilon$.
\end{abstract}
\selectlanguage{english}
\tableofcontents
\section{Introduction}
In a remarkable paper \cite{DDP11Ann}, Darmon, Dasgupta and Pollack applied the congruence between Eisenstein series and cusp forms to prove the rank one Gross conjecture for Deligne-Ribet $p$-adic $L$-functions with some assumptions, which were later removed by \cite{Ventullo15CMH}. The purpose of this paper is to apply their ideas in the setting of CM congruence to prove an analogue of Gross conjecture for the cyclotomic Katz $p$-adic $L$-functions associated with ring class characters of imaginary quadratic fields. To begin with, we let $K$ be an imaginary quadratic field and let $p>2$ be a rational prime. Fixing an isomorphism $\iota_p:\C\iso \Cp$ once and for all, let $\frakp$ be the prime above $p$ induced by $\iota_p$. We shall assume that \[p\cO_K=\frakp\pbar,\quad\frakp\neq\pbar.\]  Let $\frakf$ be a prime-to-$p$ ideal of $\cO_K$. Let $K(\frakf)$ and $K(p^\infty)$ be the ray class fields of $K$ of conductor $\frakf$ and $p^\infty$. To any $p$-adic  character $\wh\phi:\Gal(K(p^\infty)/K)\to\Cp^\x$ which is Hodge-Tate,  one can associate a character $\phi:{\rm W}_K\to\C^\x$ of the Weil group ${\rm W}_K$ of $K$ unramified outside places above $p$ with $\wh \phi(\Frob_\frakq)=\iota_p(\phi(\Frob_\frakq))$ for $\frakq\nmid p$, where $\Frob_\frakq$ denotes a geometric Frobenius at a prime $\frakq$. The character $\wh\phi$ is the $p$-adic avatar of $\phi$ (\cf \cite[page 190]{HidaTil93ASENS}). Let $\cW$ be a finite extension of the Witt ring $W(\ol{\Fp})$. Let $\chi:\Gal(K(\frakf)/K)\to \cW^\x$ be a primitive ray class character modulo $\frakf$. The works in \cite{Katz78Inv}, \cite{Shalit87book} and \cite{HidaTil93ASENS} have proved the existence of a (two-variable) Katz $p$-adic $L$-function $\cL_p(\chi)$ in the Iwasawa algebra $\cW\powerseries{\Gal(K(p^\infty)/K)}$ characterized uniquely by the following interpolation property: there exists a pair $(\Omega_p,\Omega_\infty)\in \cW^\x\times \C^\x$ such that for any $p$-adic character $\wh\phi:\Gal(K(p^\infty)/K)\to \Cp^\x$ which is crystalline of Hodge-Tate weight $(-k-j,j)$ with either $k\geq 1$ and $j\geq 0$ or $k\leq 1$ and $k+j>0$,
 \beq\label{E:interp1}\frac{\wh\phi(\cL_p(\chi))}{\Omega_p^{k+2j}}=\frac{1}{2(\sqrt{-1})^{k+j}}(1-\chi\phi(\Frob_\pbar))(1-\chi\phi(\Frob_\frakp^{-1})p^{-1})\cdot \frac{L(0,\chi\phi)}{\Omega_\infty^{k+2j}}.\eeq 
 Here $L(s,\chi\phi)$ is the \emph{complete} $L$-function of $\chi\phi$ (\cf\cite[\S 3]{Tate79Corvallis}). Let $K^+_\infty$ be the cyclotomic $\Zp$-extension of $K$. Let $\cyc:\Gal(K(p^\infty)/K)\to \Gal(K_\infty^+/K)\to \Zp^\x$ be the $p$-adic cyclotomic character. Define the cyclotomic $p$-adic $L$-function $L_p(-,\chi):\Zp\to\cW$ by 
 \[L_p(s,\chi):=\cyc^s(\cL_p(\chi)).\]
In the remainder of the introduction, we suppose that \beq\label{E:11}\chi\neq \bfone\text{ and }\chi(\Frob_\pbar)=1.\eeq
The assumption \eqref{E:11} implies that $L_p(0,\chi)=0$ by the $p$-adic Kronecker formula, and in view of Gross' conjecture for Deligne-Ribet $p$-adic $L$-functions, it is tempting to expect the leading coefficient of the Taylor expansion of $L_p(s,\chi)$ at $s=0$ to be connected with certain $\sL$-invariant, or rather $p$-adic regulator, and special values of a $L$-function. Along this direction, the work \cite{BS} provides an affirmative answer in most cases. We would like to remark that the results of \cite{BS} indeed include more general CM fields assuming some major open conjectures in algebraic number theory. We recall the (cyclotomic) $\sL$-invariant associated with $\chi$ introduced in \cite[Remark 1.5 (ii)]{BS}.  
Let $H=K(\frakf)$ be the ray class field of conductor $\frakf$ and $\cO_{H,\pbar}^\x$ be the group of $\pbar$-units. Put
\[\cO_{H,\pbar}^\x[\chi]:=\stt{u\in \cO_{H,\pbar}^\x\ot_\Z \Cp\mid (\sg\ot 1)u=(1\ot\chi(\sg))u\text{ for all }\sg\in\Gal(H/K)}.\]
We have $\cO_{H,\pbar}^\x[\chi]=\rmH^1_{\stt{\pbar}}(K,\chi^{-1}(1))$ via Kummer map.
The dimension of  the space is given by $\dim_{\Cp}\cO_{H,\pbar}^\x[\chi]=2$. Let $\frakP$ be the prime of $\cO_H$ induced by $\iota_p$. By Dirichlet's units Theorem, we can choose a basis $\stt{\ux,\vx}$ with $\ux\in \cO_H^\x\ot_\Z\Cp$ and $\vx\in \cO_{H,\ol{\frakP}}^\x\ot\Cp$ with $\Ord_{\ol{\frakP}}(\vx)\neq 0$. Let $c$ denote the complex conjugation. Define the $p$-adic logarithms $\log_\frakp,\,\log_{\pbar}:H^\x\ot_\Z\Cp\to \Cp$ by \[\log_\frakp(x\ot\alpha)=\log_p(\iota_p(x))\alpha;\quad \log_{\pbar}(x\ot\alpha):=\log_\frakp(\ol{x}\ot \alpha).\] 
Let $V_\chi$ be the kernel of $\log_\frakp:\cO_{H,\pbar}^\x[\chi]\to \Cp$. Then $V_\chi$ is a one dimensional space generated by \beq\label{E:1gen}u:=\vx\cdot (1\ot \log_{\frakp}\ux)-\ux\cdot (1\ot \log_{\frakp} \vx).\eeq
The $\sL$-invariant $\sL(\chi)$ is defined by 
\beq\label{E:Linv}\sL(\chi):=-\frac{\log_{\pbar}(u)}{\Ord_{\pbar}(u)}=\frac{1}{\log_{\frakp}\ux\cdot\Ord_{\ol{\frakP}}(\vx)}\cdot\det\pMX{\log_{\frakp} \vx}{\log_\frakp \ux}{\log_{\pbar} \vx}{\log_{\pbar} \ux}\eeq
(\cf\cite[(45), page 36]{BeD21Adv}). Note that $\sL(\chi)$ is a Gross-style regulator for imaginary quadratic fields. The above definition does not depend on the choice of basis. Let $\varphi_\frakf(\cO_K)$ be the Robert's unit in $H^\x\ot_\Z \Cp$ introduced in \cite[page 55, (17)]{Shalit87book} and put
\beq\label{E:Robert}\mathfrak e_\chi:=\sum_{\sg\in \Gal(H/K)}\sg(\varphi_\frakf(\cO_K))\ot \chi^{-1}(\sg)\in (\cO_H^\x\ot\Cp)[\chi].\eeq
We have the following Gross conjecture in the setting of imaginary quadratic fields.
\begin{conj}\label{conj:main}For all primitive ray class characters $\chi$ of $K$ modulo $\frakf$ satisfying \eqref{E:11}, we have 
 \[\frac{L_p(s,\chi)}{s}\Big\vert_{s=0}=\frac{-1}{12w_\frakf}\left(1-\frac{\chi(\Frob_\frakp^{-1})}{p}\right)\log_\frakp \mathfrak e_\chi\cdot\sL(\chi).\]
 Here $w_\frakf$ is the number of units in $\cO_K^\x$ congruent to $1$ modulo $\frakf$.
\end{conj}
When $p$ does not divide the class number of $K$, a proof of \conjref{conj:main} is given in \cite[Theorem 1.8]{BS}.  In this paper, we offer an entirely different proof of \conjref{conj:main} for ring class characters, removing the hypothesis on $p$-indivisibility of the class number. 
\begin{thma}\label{T:main}Let $d_K$ be the fundamental discriminant of $\cO_K$. Suppose that $\chi$ is a ring class character and that $(\frakf,d_K)=1$. Then \conjref{conj:main} holds. \end{thma}
Regarding the non-vanishing of $\sL$-invariants, we remark that it is shown in \cite[Proposition 1.11]{BeD21Adv} that either $\sL(\chi)$ or $\sL(\chi^{-1})$ is non-zero and that the $\sL$-invariant $\sL(\chi)$ is non-zero if the Four Exponentials Conjecture holds.

The proof in \cite{BS} requires the full arsenal of Iwasawa theory for imaginary quadratic fields and the existence of elliptic units. In the case of general CM fields, their method relies on the existence of Rubin-Stark units in ray class fields, which is one of the major open conjectures in algebraic number theory. In contrast, we adapt the ideas in \cite{DDP11Ann}, replacing the Eisenstein congruence with the CM congruence for elliptic modular forms. This approach is inspired by a series of works of Hida and Tilouine \cite{HidaTil91Durham}, \cite{HidaTil93ASENS} and \cite{HidaTil94Inv} on CM congruences and the anticyclotomic main conjecture for CM fields and a recent work \cite{BeD21Adv}. This method is \emph{units-free} and more amenable to general CM fields as in \cite{DDP11Ann} at least under some suitable Leopoldt conjecture (see \cite{BHsi} for the case of general CM fields). We now give a sketch of the proof of \thmref{T:main}.\subsection*{Cohomological interpretation of the $\sL$-invariant}
Following the discussion on the anticyclotomic $\sL$-invariants in \cite[\S 1.4]{BeD21Adv}, the staring point is the observation $\dim_F\rmH^1(K,\chi)=1$ by the global Poitou-Tate duality. Let $\kappa\neq 0\in \rmH^1(K,\chi)$. Write $\loc_\pbar(\kappa)\in \rmH^1(K_\pbar,\chi)=\Hom(G_{K_\pbar},\Cp)$. Let $\kappa_{\rm ur}:G_{K_\pbar}\to E$ be the unique unramified homomorphism sending the geometric Frobenius $\Frob_\pbar$ to $1$ and $\kappa_{\rm cyc}$ be the $p$-adic logarithm of the $p$-adic cyclotomic character. Then
\[\loc_{\pbar}(\kappa)=x\cdot \kappa_{\rm ur}+y\cdot \kappa_{\rm cyc}.\]
In \lmref{L:41}, we show that $y\neq 0$ and
\beq\label{E:10}\sL(\chi)=\frac{x}{y}.\eeq
Therefore, to prove \thmref{T:main} we need to construct a non-zero cohomology class $\kappa$ whose $x$ and $y$ coordinates can be evaluated explicitly and related to the derivatives of the Katz $p$-adic $L$-functions.

\subsection*{$\sA$-adic modular forms and construction of cohomology classes}The construction of $\kappa$ relies on the idea in \cite{HidaTil94Inv} of using the congruence between $p$-adic families of CM forms and non-CM forms to prove anticyclotomic main conjectures.  Let $\sA$ be the ring of rigid analytic functions on the closed unit disk $\stt{s\in\Cp\mid \abs{s}_p\leq 1}$. For an integer $k\geq 1$, a prime-to-$p$ positive integer $N$ and a Dirichlet character $\xi$ modulo $N$, let $S_k(N,\chi)$ be the space of elliptic cusp forms of weight $k$, level $N$ and character $\xi$. Denote by $\bfS(N,\xi)$ the space of ordinary $\sA$-adic modular forms of tame level $N$ and character $\xi$, consisting of $q$-expansion $\cF(s)(q)\in \sA\powerseries{q}$ such that $\cF(k)(q)$ is the $q$-expansion of some $p$-ordinary elliptic cusp form of weight $k+1$, level $Np$ for all but finitely many $k\con 0\pmod{p-1}$. Since $\chi$ is assumed to be a ring class character, we can write $\chi=\brch^{1-c}$ for some ray class character $\brch$ of conductor $\frakc$ prime to $d_Kp$. Note that the choice of $\brch$ is not unique. Let $N=d_K\rmN\frakc$ and $\xi:=\tau_{K/\Q}\brch_+$, where $\tau_{K/\Q}:(\Z/d_K\Z)^\x\to\C^\x$ is the quadratic character associated with $K/\Q$ and $\brch_+:(\Z/\rmN\frakc\Z)^\x\to\C^\x$ is given by $\brch_+(a)=\brch(a\cO_K)$. Let $\bftheta_\brch$ and $\bftheta_{\brch^c}$ be $\sA$-adic CM forms in $\bfS(N,\brch_+\tau_{K/\Q})$ associated with $\brch$ and $\brch^c$ defined in \eqref{E:24}. Let $\sK=\Frac\sA$. The theory of $\sA$-adic newforms yields a decomposition of Hecke modules 
\beq\label{E:40}\bfS(N,\brch_+\tau_{K/\Q})=\sK\bftheta_\brch\oplus\sK\bftheta_{\brch^c}\oplus \bfS^\perp.\eeq
The submodule $\bfS^\perp$ interpolates the orthogonal complement of the space spanned by $\bftheta_\brch$ and $\bftheta_{\brch^c}$.
Let $\bfT^\perp$ be the $\sA$-algebra generated by the Hecke operators acting on $\bfS^\perp$. 
Suppose we are given a Hecke eigensystem $\lam:\bfT^\perp\to\sA^\dagger/(s^2)$ and a character $\Psi:G_K\to \sA^\dagger/(s^2)$ such that 
\begin{itemize}\item[(a)] $\Psi\con \phi\pmod{s}$,
\item[(b)] $\lam(T_\ell)=\Psi(\Frob_\frakl)+\Psi(\Frob_{\ol{\frakl}})$ for $\ell=\frakl\ol{\frakl}$ split in $K$.
\end{itemize}
Write $\Psi=\phi(1+\psi' s)\pmod{s^2}$. In \thmref{T:Ribet}, we use the argument in \cite[\S 4]{DDP11Ann} to construct a non-zero cohomology class $\kappa\in \rmH^1(K,\chi)$ such that \beq\label{E:20}\loc_{\pbar}(\kappa)=\psi'|_{G_{K_{\pbar}}}-\brch(\pbar)^{-1}\lam(U_p)'(0)\cdot \kappa_{\rm ur}.\eeq
Here $\lam(U_p)'(0)$ is the first derivative of the $U_p$-eigenvalue $\lam(U_p)$ at $s=0$.

\subsection*{Construction of Hecke eigenforms modulo $s^2$}The problem boils down to constructing a Hecke eigensystem $\lam:\bfT^\perp\to\sA^\dagger/(s^2)$ as above and computing the derivative of the $U_p$-eigenvalue $\lam(U_p)$. This is the main bulk of this paper and is achieved by applying the $p$-adic Rankin-Selberg method. For any $C\mid N$, let $\sG_C(s)\in \sA\powerseries{q}$ be the $q$-expansion defined by \[\sG_C(s)=1+2\zeta_p(1-s)^{-1}\sum_{n=1
}^\infty\left(\sum_{d\mid n,p\ndivides n}d^{-1}\Dmd{d}^{s}\right)q^{Cn},\]
where $\zeta_p(s)$ is the $p$-adic Riemann zeta function. For $k\geq 2$, $\sG_C(k)$ is the $q$-expansion of an $p$-ordinary Eisenstein series of weight $k$ and level $\Gamma_0(Cp)$. From the spectral decomposition of  $e_\Ord(\theta_\brch^\circ\sG_C)\in \bfS(N,\brch_+\tau_{K/\Q})$ in \eqref{E:40}, we find that there exist $ \aA$ and $\bB$ in $\sK$ such that \[e_\Ord(\theta_\brch^\circ\sG_C)=\aA\bftheta_{\brch}+\bB\bftheta_{\brch^c}+\sH\quad\]
for some $\sA$-adic form $\sH\in\bfS^\perp$. According to \cite[Theorem 8.1]{HidaTil93ASENS}, the coefficients $\aA$ and $\bB$ are essentially a product of two-variable Katz $p$-adic $L$-functions $\cL_p(s,t,\chi)$ (See \eqref{E:2Katz} for the definition).  By Hida's $p$-adic Rankin-Selberg method, we will prove in \propref{P:3RS} the following precise identity \beq\label{E:30} \bB(s)=\frac{2\cL_p(s,0,\bfone)\cL_p(s,0,\chi)}{L(0,\tau_{K/\Q})\cL_p(s,-s,\chi)\zeta_p(1-s)}\cdot\frac{\Dmd{d_K}^s}{(1-\e_{\frakp}^s(\Frob_{\pbar}))^2}\eeq
 for some good choices of $\brch$ and $C$. Following a similar calculation in \cite[\S 3]{Ventullo15CMH},  we will see in \thmref{T:43} that the $\sA$-adic form $\sH$ produces an explicit Hecke eigensystem $\lam_\sH:\bfT^\perp\to \sA^\dagger/(s^2)$ with the properties (a) and (b) and use \eqref{E:30} to show that the first derivative of $\lam_\sH(U_p)$ is given by the derivatives of the Katz $p$-adic $L$-functions.  Putting all ingredients together, we prove \thmref{T:main} in \subsecref{SS:43}. 
 
 Finally, in \secref{S:comparison} we compare the definition of $\sL$-invariants in \eqref{E:Linv} and Benois' $\sL$-invariant in the setting of imaginary quadratic fields. First, Perrin-Riou \cite{PR95Ast229} formulated a general conjecture for special values of $p$-adic $L$-functions at all integer points except for the exceptional zero case. Using an idea of Greenberg \cite{Gr94} in the ordinary case, Benois \cite{B3} gave a general definition of $\sL$-invariant using $(\varphi,\Gamma)$-modules and formulated a trivial zero conjecture including the non-critical case.
We confirm that our formula is compatible with his conjecture in \S 5.2.

\subsection*{Notation and convention}If $F$ is a local or global field of characteristic zero, let $\cO_F$ be the ring of integers of $F$. 
Let $G_F$ denote the absolute Galois group of $F$ and let $C_F:=F^\times$ if $F$ is local and $C_F$ be the idele class group $\A_F^\x/F^\x$ if $F$ is global. Let $\rec_F:C_F\to G_F^{ab}$ be the \emph{geometrically normalized} reciprocity law homomorphism. 

Let $F$ be a global field. If $\frakq$ is a prime ideal of $\cO_F$ (resp. $v$ is a place of $F$), 
let $F_\frakq$ (resp. $F_v$) be the completion of $F$ at $\frakq$ (resp. $v$). 
Then $\rec_{F_\frakq}:F_\frakq^\x\to G_{F_\frakq}^{ ab}$ sends a uniformizer $\uf_\frakq$ of $\cO_{F_\frakq}$ to the corresponding geometric Frobenius $\Frob_\frakq$. If $S$ is a finite set of prime ideals of $\cO_F$, let $F_S$ be the maximal algebraic extension of $F$ unramified outside $S$ and let $G_{F,S}=\Gal(F_S/F)$. For a fractional ideal $\fraka$ of a global field $F$, we let $\Frob_\fraka:=\prod_{\frakq}\Frob_\frakq^{n_\frakq}$ if $\fraka$ has the prime ideal factorization $\prod_\frakq \frakq^{n_\frakq}$. 

If $\chi:C_F\to\C^\x$ is an idele class character of $F^\x$ unramified outside $S$. If $v$ is a place of $F$, let $\chi_v:F_v^\x\to\C^\x$ be the local component of $\chi$ at $v$ and let $L(s,\chi_v)$ be the local $L$-factor of $\chi_v$ in \cite[(3.1)]{Tate79Corvallis}. Let $L(s,\chi)=\prod_v L(s,\chi_v)$ is the \emph{complete} $L$-function of $\chi$ and $\epsilon(s,\chi)$ be the epsilon factor  \cite[(3.5.1-2)]{Tate79Corvallis}. If $\chi=\bfone$ is the trivial character, then we put $\zeta_{F_v}(s)=L(s,\bfone_v)$ and $\zeta_F(s)=L(s,\bfone)$. In particular, if we denote by $\zeta(s)$ the usual Riemann zeta function, then $\zeta_\Q(2)=\zeta_\R(s)\zeta(s)=\pi^{-\frac{s}{2}}\Gamma(\frac{s}{2})\zeta(s)$ and hence $\zeta_\Q(2)=\pi/6$ under our definition of the complete $L$-function.
If $\chi$ is a character of $G_{F,S}$, we shall view $\chi$ as a Hecke character of $C_F$ via $\rec_F$ and still denote by $\chi$ if there is no fear for confusion. Therefore, \[\chi(\frakq):=\chi(\Frob_\frakq)=\chi_\frakq(\uf_\frakq)\text{ for }\frakq\not\in S.\]
In particular, a primitive ray class character $\chi$ modulo $\frakc$ shall be identified with an idele class character $\chi$ of $F$ of conductor $\frakc$. 

We write $\A=\A_\Q$ for simplicity. Denote by $\bfe=\prod\bfe_v:\A/\Q\to\C^\x$ the unique additive character with $\bfe_\infty(x)=\exp(2\pi\sqrt{-1}x)$. If $\chi:\A^\x/\Q^\x\to\C^\x$ is a finite order idele class character of $\Q$ of level $N$, then let  $\chi_{\rm Dir}$ be the Dirichlet character modulo $N$ obtained by the restriction of $\chi$ to $\prod_{\ell\mid N}\Z_\ell^\x$. With our convention, if $q\ndivides N$ is a prime, then 
 \beq\label{E:13}\chi_q(q)=\chi((q))=\chi_{\rm Dir}(q)^{-1}.\eeq 
 We fix an isomorphism $\iota_p:\C\iso\Cp$ once and for all. Let $\Om:\Gal(\Q(\zeta_p)/\Q)\to\C^\x$ be Galois character such that $\iota_p\circ\Om$ is the $p$-adic \Teich character. Identifying $\Om$ with an idele class character of $\Q$, we have \[\iota_p(\Om_{\rm Dir}(a))\con a\pmod{p};\quad L(s,\Om)=L(s,\Om_{\rm Dir}^{-1}).\]

 \section{Ordinary $\Lam$-adic CM forms}
\subsection{Ordinary $\Lam$-adic forms}
If $N$ is a positive integer, let $\cS_k(N,\chi)$ denote the space of elliptic cusp forms of level $\Gamma_1(N)$ and Nebentypus $\chi_{\rm Dir}^{-1}$. If $f\in \cS_k(N,\chi)$ is a Hecke eigenform, let $\varphi_f:=\varPhi(f)$ be the associated automorphic form. Let $\Q_\infty$ be the cyclotomic $\Zp$-extension of $\Q$ and $\Gamma_\Q=\Gal(\Q_\infty/\Q)$. Define the Iwasawa algebra $\Lam:=\cW\powerseries{\Gamma_\Q}$ and write $\sigma\mapsto [\sigma]$ for the inclusion of group-like elements $\Gamma_\Q\to \Lam^\x$. If $\nu:\Gamma_\Q\to\Cp^\x$ is a continuous character, we extend $\nu$ uniquely to a $\cW$-algebra homomorphism $\nu:\Lam\to\Cp^\x$ by the formula $\nu([\gamma])=\nu(\gamma)$. Let $\cyc:\Gamma_\Q\to 1+p\Zp$ be the cyclotomic character. For $s\in \Zp$, let $P_s$ be the kernel of $\cyc^s:\Lam\to\Zp$, \ie the ideal of $\Lam$ generated by $\stt{[\sg]-\cyc^s(\sg)\mid \sg\in\Gamma_\Q}$. For a positive prime-to-$p$ integer $N$ and a finite order idele class character $\chi$ modulo $pN$, let $\bfS^\Ord(N,\chi,\Lam)$ be the space of (ordinary) $\Lam$-adic cusp forms of tame level $N$ with Nebentypus $\chi_{\rm Dir}^{-1}$, consisting of $q$-expansions $\cF(q)=\sum_{n}\bfa(n,\cF)q^n\in \Lam\powerseries{q}$ such that for $k\geq 1$, the specialization $\cF\pmod{P_k}=\sum_{n}\cyc^k(\bfa(n,\cF))q^n$ is the $q$-expansion of some cusp form $\cF_k$ in $\cS_{k+1}^\Ord(pN,\chi\Om^{k})\ot_{\C,\iota_p}\Cp$ at the infinity cusp. 

If $R$ is a $\Lam$-algebra which is an integral domain and finite over $\Lam$, let $\bfS^\Ord(N,\chi,R):=\bfS^\Ord(N,\chi,\Lam)\ot_{\Lam} R$ be the space of $\Lam$-adic forms defined over $R$. A basic result in Hida theory asserts that $\bfS^\Ord(N,\chi,R)$ is a free $R$-module of finite rank equipped with the action of Hecke operators $\stt{T_\ell}_{\ell\ndivides pN},\stt{U_q}_{q\mid pN}$. We let \[\bfT(N,\chi,R)=R[\stt{T_\ell}_{\ell\ndivides pN},\stt{U_q}_{q\mid pN}]\subset \End_{R}\bfS^\Ord(N,R,\chi)\]be the big ordinary cuspidal Hecke algebra generated by these Hecke operators over $R$. By the freeness of $\bfS^\Ord(N,\chi,R)$, we have $\bfT(N,\chi,R)=\bfT(N,\chi,\Lam)\ot_{\Lam} R$. A prime ideal $Q$ in $\Spec R$ is called an arithmetic point if $Q$ is lying above $P_k$ for some $k\geq 2$. A $\Lam$-adic form $\cF$ in $\bfS^\Ord(N,\chi,R)$ is a \emph{newform} of tame level $N_\cF\mid N$ if for all but finite many arithmetic primes $Q$ of $\Spec R$, the specialization $\cF\pmod{Q}\in \cS_{k+1}^\Ord(pN_\cF,\chi\Om^{k})$ is the $q$-expansion of a $p$-stabilized normalized elliptic newform of tame level $N_\cF$. 
\subsection{Classical CM forms}Let $K$ be an imaginary quadratic field and let $d_K>0$ be the fundamental discriminant of $\cO_K$. Let $\tau_{K/\Q}:(\Z/d_K\Z)^\x\to\stt{\pm 1}$ be the quadratic character associated with $K/\Q$. If $\psi$ is an idele class character of $K$ of conductor $\frakc$ with $\psi_\infty(z)=z^{-k}$ for some non-negative integer $k$, we recall that the CM form associated with $\psi$ is the elliptic modular form $\theta_\psi^\circ$ of weight $k+1$ defined by the $q$-expansion
\[\theta_\psi^\circ=\sum_{(\fraka,\frakc)=1}\psi(\fraka)q^{\rmN\fraka},\]
where $\fraka$ runs over ideals of $\cO_K$ prime to $\frakc$ and $\rmN\fraka:=\rmN_{K/\Q}(\fraka)$ is the norm of $\fraka$. Write $\psi_+:=\psi|_{\A^\x_\Q}=\Abs_{\A_\Q}^k\om$ for some finite order idele class character $\om$ of $\Q$. Then $\theta_\psi^\circ$ is a newform of weight $k+1$, level $\rmN \frakc d_K$ and Nebentypus $\om_{\rm Dir}^{-1}\tau_{K/\Q}$. 
Let $\frakp$ be a prime of $\cO_K$ lying above $p$. 
The $\frakp$-stabilization $\theta_\psi^{(\frakp)}$ is defined by 
\[\theta_\psi^{(\frakp)}=\sum_{(\fraka,\frakp\frakc)=1}\psi(\fraka)q^{\rmN \fraka}.\]
\subsection{$\Lam$-adic CM forms}
Suppose that $p\cO_\cK=\frakp\ol{\frakp}$, where $\frakp$ is the prime induced by the fixed embedding $\Qbar\hookto\C\iso\Qbarp$. Let $K_{\frakp^\infty}$ be the $\Zp$-extension of $K$ in $K(\frakp^\infty)$ and $\Gamma_{K,\frakp}=\Gal(K_{\frakp^\infty}/K)$. Let $\frakc$ be an ideal of $\cO_K$ coprime to $p$. The transfer map $\sV:G_\Q^{ab}\to G_K^{ab}$ induces a map $\sV:\Gamma_\Q\to \Gal(K(p^\infty)/K)\to \Gamma_{K,\frakp}$, which in turns gives rise to an embedding \beq\label{E:22}\sV:\Lam=\cW\powerseries{\Gamma_\Q}\to\Lam_K:=\cW\powerseries{\Gamma_{K,\frakp}}\eeq such that $\sV(\rec_{\Qp}(z)|_{\Q_\infty})=\rec_{K_\frakp}(z)|_{K_{\frakp^\infty}}$ for $z\in\Qp^\x$. Let $\Psi^{\rm univ}:G_K\to \Lam_K^\x$ be the universal character defined by the inclusion of group-like elements $\Gamma_{K,\frakp}\to \Lam_K^\x$
\beq\label{E:23}\Psi^{\rm univ}(\sg)=[\sg^{-1}|_{K_{\frakp^\infty}}]\in\Lam_K.\eeq
For any primitive ray class character $\brch$ modulo $\frakc$, we define 
\beq\label{E:24}\bftheta_\brch(q)=\sum_{(\fraka,\frakp\frakc)=1}\brch(\fraka)\cdot\Psi^{\rm univ}(\Frob_\fraka)q^{\rmN \fraka}\in \Lam_K\powerseries{q}.\eeq
Let $\brch_+=\brch\circ\sV$, regarded as an idele class character of $\Q$. 
Then $\bftheta_\brch$ is a $\Lam$-adic newform of tame level $N:=d_K\rmN \frakc$ and Nebentypus $\brch_+\tau_{K/\Q}$. Let $\bfS:=\bfS^\Ord(N,\brch_+\tau_{K/\Q},\Lam_K)$ and $\bfT:=\bfT(N,\brch_+\tau_{K/\Q},\Lam_K)$.  Then $\bfS$ is a free $\Lam_K$-module with $\bfT$-action. Denote $\bfK = \operatorname{Frac} (\Lam_K)$.
Let $\bfS^\perp$ be the subspace of $\bfS\ot\bfK$ generated by the following set \[\Xi^\perp=\stt{\cF(q^M)\mid \cF\neq \bftheta_\brch\text{ or }\bftheta_{\brch^c}\text{ a newform in $\bfS$ of tame level $N_\cF$}\text{ and }MN_\cF\mid N}.\] 
By the theory of $\Lam$-adic newforms \cite[Proposition 1.5.2]{Wiles88Inv}, we have the decomposition of $\bfT$-modules
 \beq\label{E:decomp.1}\bfS\ot_{\Lam_K}\bfK=\bfK\cdot\bftheta_\brch\oplus\bfK\cdot\bftheta_{\brch^c}\oplus \bfS^\perp.\eeq

\section{The $p$-adic Rankin-Selberg convolutions}
\subsection{A classical Eisenstein series}We recall a general construction of Eisenstein series in the theory of automorphic forms. If $\om$ is a finite order idele class character of $\Q$ and $k$ is an integer, let $\cA_k(\om)$ denote the space of automorphic forms $\varphi:\GL_2(\Q)\bksl \GL_2(\A)\to\C$ such that 
\[\varphi(zg\kappa_\theta)=\om(z)\varphi(g)e^{2\pi\sqrt{-1}k\theta},\quad z\in \A^\x,\,\kappa_\theta=\pMX{\cos\theta}{\sin\theta}{-\sin\theta}{\cos\theta}\in\SO_2(\R).\]
Let $\cA^0_k(\om)\subset \cA_k(\om)$ be the subspace of cusp forms. For $a\in\Zp^\x$, put $\Dmd{a}:=a\cdot (\iota_p \circ\Om )(a)^{-1}$.
For each place $v$, let $\Om_v$ be the local component of $\Om$ at $v$.
Let $\cD$ be the pair 
\[\cD=(k,C),\quad C\in\Z_{>0}\text{ and }p\ndivides C.\] Let $\cS(\A^2)$ be the space of \BS functions on $\A^2$. Define 
$\Phi_{\cD}=\Phi_{\cD,\infty}\ot_{\ell}'\Phi_{\cD,\ell}\in \cS(\A^2)$ by 
\begin{itemize}
\item $\Phi_{\cD,\infty}(x,y)=2^{-k}(x+\sqrt{-1}y)^ke^{-\pi(x^2+y^2)}$, 
\item $\Phi_{\cD,\ell}(x,y)=\bbI_{C\Z_\ell}(x)\bbI_{\Z_\ell}(y)$,
\item $\Phi_{\cD,p}(x,y)=
\Om_p^{-k}(x)\bbI_{\Zp^\x}(x)\bbI_{\Zp}(y).$\end{itemize}
Recall that $f_{\cD,s}=\ot_v f_{\cD,s,v}$, where $f_{\cD,s,v}=f_{\Om_v^k,\bfone,\Phi_{\cD,v},s}:\GL_2(\Q_v)\to\C$ is the Godement section associated with $\Phi_{\cD,v}$ defined by 
\[f_{\cD,s,v}(g_v)=\Om_v^k(\det g)\abs{\det g_v}_v^{s+\onehalf}\int_{\Q_v^\x}\Phi_{\cD,v}((0,t_v)g_v)\Om^k(t_v)\abs{t_v}_v^{2s+1}\rmd^\x t_v\]
 (\cf\cite[(4.1)]{HsiehChen20}). Let $B(\Q)$ be the upper triangular matrices in $\GL_2(\Q)$. Then the Eisenstein series $E_\A(-,f_{\cD,s}):\GL_2(\A)\to\C$ is the series defined by 
\[E_\A(g,f_{\cD,s})=\sum_{\gamma\in B(\Q)\bksl \GL_2(\Q)}f_{\cD,s}(\gamma g)\in \cA_k(\bfone)\]
(\cf\cite[(7.8), page 351]{BumpGrey}). The series $E_\A(g,f_{\cD,s})$ is absolutely convergent for $\Re(s)>1/2$ and can be analytically continued to the whole complex plane except at $s=\pm\frac{1}{2}$. Suppose that $k\geq 2$. For $z=x+\sqrt{-1}y\in \frakH=\stt{z\in\C\mid \Im(z)>0}$, put \[E_k(C)(z)=y^{-\frac{k}{2}}E_\A(\pMX{y}{x}{0}{1},f_{\cD,s})|_{s=\frac{1-k}{2}}.\]
Then $E_k(C)(z)$ defines a classical Eisenstein series of weight $k$ and level $\Gamma_0(pC)$.
\begin{prop}\label{P:3FE}The Fourier expansion of $E_k(C)$ is given by
\[E_k(C)=\frac{\Dmd{C}^k}{2C}\zeta_p(1-k)+\sum_{n>0,C\mid n}\bfa(n,E_k(C))q^n,\]
where
\beq\label{E:35}\bfa(n,E_k(C))= \sum_{C\mid d\mid n,\,p\ndivides d}d^{k-1}\Om_{\rm Dir}(d)^{-k}.\eeq
\end{prop}
\begin{proof}For each positive integer $n$, the Fourier coefficient $\bfa(n,E_k(C))$ is the product of local Whittaker functions 
\[\bfa(n,E_k(C))=n^\frac{k}{2}\prod_{\ell}W(\pDII{n}{1},f_{\cD,s,\ell})|_{s=\frac{1-k}{2}},\] 
where $W(-,f_{\cD,s,\ell}):\GL_2(\Q_\ell)\to\C$ is the local Whittaker function defined by 
\[W(g,f_{\cD,s,\ell})=\lim_{n\to\infty}\int_{\ell^{-n}\Z_\ell}f_{\cD,s,\ell}(\pMX{0}{-1}{1}{0}\pMX{1}{x}{0}{1}g)\bfe_\ell(-x)\rmd x_\ell,\]
and the Haar measure $\rmd x_\ell$ is normalized so that $\vol(\Z_\ell,\rmd x_\ell)=1$ (\cf\cite[Corollary 4.7]{HsiehChen20} and \cite[(7.14)]{BumpGrey}). Hence we get \eqref{E:35} from the explicit formulae of these local Whittaker functions in \cite[Lemma 4.6]{HsiehChen20}. 

On the other hand, the constant term $\bfa(0,E_k(C))$ of $E_k(C)$ at the infinity cusp is given by 
\[\bfa(0,E_k(C))=f_{\cD,\frac{1-k}{2}}(1)+(Mf_{\cD,s})(1)\Big\vert_{s=\frac{1-k}{2}},\]
where $Mf_{\cD,s}(g)$ is obtained by the analytic continuation of the intertwining integral
\[Mf_{\cD,s}(g)=\int_{\A}f_{\cD,s}(\pMX{0}{-1}{1}{0}\pMX{1}{x}{0}{1})\rmd x,\,g\in \GL_2(\A)\]
(\cf\cite[(7.15)]{BumpGrey}). A direct computation shows that for $\Re(s)\gg 0$, \begin{align*}Mf_{\cD,s}(1)=&\prod_v\int_{\Q_v}f_{\cD,s,v}(\pMX{0}{-1}{1}{x})\rmd x_v\\
 =&\frac{C^{2s}\prod_{\ell\mid C}\Om^k_\ell(C)}{2}\cdot \frac{L(2s,\Om^k)}{L(2s,\Om_p^k)}\\
 =&\frac{C^{2s}\Om((C))^{k}}{2}\cdot L(2s,\Om^k)\cdot 
 \begin{cases} (1-p^{-2s}\Om_p^{k}(p))&\text{ if $\Om_p^k$ is unramified},\\
 1& \text{ if $\Om_p^k$ is ramified},\end{cases}
\end{align*}
(\cf \cite[Proposition 2.6.3 and (7.27)]{BumpGrey}). Since $f_{\cD,s,p}(1)=0$, we see that 
\begin{align*}\bfa(0,E_k(C))&=Mf_{\cD,s}(1)|_{s=\frac{1-k}{2}}\\=&\frac{\Dmd{C}^k}{2C}(1-p^{k-1}\Om_{\rm Dir}^{-k}(p))L(1-k,\Om_{\rm Dir}^{-k})=\frac{\Dmd{C}^k}{2C}\zeta_p(1-k).\end{align*}
This finishes the computation of the Fourier expansion of $E_k(C)$. \end{proof}
\begin{Remark}Let $E_k(z)$ be the standard classical Eisenstein series with the $q$-expansion
\[E_k=\frac{\zeta(1-k)}{2}+\sum_{n>0}\sigma_{k-1}(n)q^n.\]
Let $E^{(p)}_k(z):=E_k(z)-p^{k-1}E_k(pz)$ be the $p$-stabilization of $E_k$. From the inspection of Fourier expansions, we have 
\[E_k(C)(z)=C^{-1}\Dmd{C}^k\cdot E^{(p)}_k(Cz).\]
The adelic construction of $E_k(C)$ will be used in the later computation of the adelic Rankin-Selberg convolution. 
\end{Remark}
\subsection{A $\Lam$-adic Eisenstein series}
Let $P$ be the augmentation ideal of $\Lam$. Let $\cL^{\rm KL}_p(\bfone)\in
P^{-1}\Lam$ be the Kubota-Leopoldt $p$-adic $L$-function associated with trivial character, \ie $\cyc^s(\cL^{\rm KL}_p(\bfone))=\zeta_p(1-s)$. Define the $q$-expansion
\begin{align*}&\cE_C:=\frac{[\Frob_C]^{-1}}{2C}\cdot \cL^{\rm KL}_p(\bfone)+\sum_{n>0,\,C\mid n}\bfa(n,\cE_C) q^n,\\&\bfa(n,\cE_C)=\sum_{C\mid d\mid n,\,p\ndivides d}d^{-1}[\Frob_d]^{-1}\in \Lam.\end{align*}
\begin{prop}The $q$-expansion $\cE_C$ defines a $\Lam$-adic form of Eisenstein series. More precisely, for $k\geq 2$, we have
\[\cyc^k(\cE_C)=E_k(C)(q).\]
\end{prop}
\begin{proof}Note that with our convention \eqref{E:13}, for any positive integer $a$ prime to $p$,  $\Frob_a$ is an element in $G_\Q$ corresponding to the ideal $(a)=a\Z$ and \[\cyc(\Frob_a)=\Dmd{a}^{-1}=a^{-1}\Om_{\rm Dir}(a).\] 
The assertion thus follows from \propref{P:3FE} immediately.
\end{proof}

\subsection{Two-variable and improved Katz $p$-adic $L$-functions}
Let $\frakf$ be an integral ideal of $K$. If $\chi$ is an idele class character of $K$ with the conductor $\frakf$.  The (finite) Hecke $L$-function for $\chi$ is defined by the Dirichlet series
\[L_{\rm fin}(s,\chi)=\sum_{(\fraka,\frakf)=1}\chi(\fraka)\rmN \fraka^{-s}.\]
If the infinity type of $\chi$ is $(a,b)\in\Z^2$, \ie $\chi_\infty(z)=z^a\ol{z}^b$, then the Hecke $L$-function associated with $\chi$ is given by 
\beq\label{E:Lfcn}L(s,\chi):=2(2\pi)^{-(s+\max\stt{a,b})}\Gamma(s+\max\stt{a,b})L_{\rm fin}(s,\chi)\eeq

Suppose that $(\frakp\pbar,\frakf)=1$. We consider the $p$-adic $L$-functions $L_{p,\frakf\pbar^\infty}$ and $L_{p,\frakf}$ of $K$ defined in \cite[(49), page 86]{Shalit87book}. Let $\chi$ be a primitive ray class character modulo $\frakf$. Let $\cL_p(\chi)$ be the unique element in the Iwasawa algebra  $\cW\powerseries{\Gal(K(p^\infty)/K}$ such that for every $p$-adic continuous character $\ep$ on $\Gal(K(p^\infty)/K)$, we have \[\ep(\cL_p(\chi))=L_{p,\frakf\pbar^\infty}(\chi\ep)\ep(\sigma_\delta),\] where $\sigma_\delta\in \Gal(K(\frakf p^\infty)/K(\frakf \pbar^\infty))$ is the element defined in \cite[(7), page 92]{Shalit87book}. 
 We call $\cL_p(\chi)$ the two-variable Katz $p$-adic $L$-function associated with $\chi$. Let $\e_\frakp: \Gamma_{K,\frakp}=\Gal(K_{\frakp^\infty}/K)\to \cW^\times$ be a $p$-adic character such that
 \[\e_\frakp(\rec_K(z))=\Dmd{z_{\frakp}},\,z\in \wh\cO_K^\x.\]
 By definition, $\e_\frakp\circ\sV=\cyc$. Let $\e_{\pbar}(\sg):=\e_{\frakp}(c\sg c)$. It is convenient to introduce the two-variable Katz $p$-adic $L$-function $\cL_p(s,t,\chi):\Zp^2\to \cW$ defined by 
 \beq\label{E:2Katz}\cL_p(s,t,\chi):=(\e_\frakp^{s}\e_{\pbar}^t)\left(\cL_p(\chi)\right)\text{ for } (s,t)\in\Zp^2.\eeq
Let $\psi$ be the idele class character of $K^\x$ such that $\wh\psi=\e_{\frakp}$, \ie $\psi:\A^\x_K/K^\x\to \C^\x$ is an idele class character of $K$ unramified outside $\frakp$ and $\psi_\infty(z)=z$ and $\psi(\frakq)=\psi_\frakq(\uf_\frakq)=\e_\frakp(\Fr_\frakq)$ for any prime $\frakq\neq \frakp$. 
\begin{prop}\label{P:3interpolation}There exists periods $(\Omega_\infty,\Omega_p)\in \C^\x\times \cW^\x$ such that for all $(k,j)\in\Z^2$ such that $k\geq 1$ and $j\geq 0$ or $k\leq 1$ and $k+j>0$, we have
 \[\frac{\cL_p(k+j,-j,\chi)}{\Omega_p^{k+2j}}=\frac{1}{2(\sqrt{-1})^{k+j}}(1-\chi\psi^{k+j(1-c)}(\pbar))(1-\chi\psi^{k+j(1-c)}(\frakp^{-1})p^{-1})\frac{L(0,\chi\psi^{k+j(1-c)})}{\Omega_\infty^{k+2j}}.\]
 \end{prop}
 \begin{proof}Let $(\Omega,\Omega_p)\in\C^\x\times \cW^\x$ be the periods introduced in \cite[Theorem 4.14, page 80]{Shalit87book} and put $\Omega_\infty:=(2\pi)^{-1}\Omega\sqrt{d_K}$. Write $\varepsilon=\e_\frakp^{k+j}\e_\pbar^{-j}$. One deduces the desired interpolation formula of $\cL_p(k+j,-j,\chi)=L_{p,\frakf\pbar^\infty}(\chi\varepsilon)\varepsilon(\sigma_\delta)$ from \cite[(50), page 86 and Lemma (i), page 92]{Shalit87book}.
\end{proof}
 Likewise we define $\cL_\frakp^*(\chi)$ to be the unique element in $\cW\powerseries{\Gal(K(\frakp^\infty)/K)}$ such that 
 \[\ep(\cL_\frakp^*(\chi))=L_{p,\frakf}(\chi\ep)\ep(\sigma_\delta)\]
 for any $p$-adic character $\ep$ on $\Gal(K(\frakp^\infty)/K)$. Put \[\cL_\frakp^*(s,\chi):=\e_\frakp^s(\cL_\frakp^*(\chi)).\]
 Then $\cL_\frakp^*(\chi)$ is called the (one-variable) improved $p$-adic $L$-function associated with $\chi$ in the sense that
 \beq\label{E:improved}\cL_p(s,0,\chi)=(1-\chi(\pbar)\e_{\frakp}^s(\Frob_{\pbar}))\cL_\frakp^*(s,\chi).\eeq
If $\chi\neq\bfone$, then by the $p$-adic Kronecker limit formula \cite[Theorem 5.2, page 88]{Shalit87book}, we have \beq\label{E:improved2}\cL^*_\frakp(0,\chi)=\frac{-1}{12w_{\frakf}}\left(1-\frac{\chi(\frakp^{-1})}{p}\right)\log_p \frake_\chi,\eeq
where $\frake_\chi$ is the Robert's unit in \eqref{E:Robert}. 
It follows that $\cL^*_\frakp(0,\chi)\neq 0$ by the Brumer-Baker Theorem.

 \begin{Remark}\label{R:34}Recall that the cyclotomic $p$-adic $L$-function $L_p(s,\chi):=\cyc^s(\cL_p(\chi))$. Let $h$ be the class number of $K$. Since $\e_{\frakp}^h\e_{\pbar}^h=\cyc^h$, we have
 \[\cL_p(ht,ht,\chi)=L_p(ht,\chi)\text{ for }t\in\Zp.\]
\end{Remark} 
\subsection{Rankin-Selberg convolution with CM forms}
Let $\chi$ be a ring class character unramified outside $pd_K$. There exists a ray class character $\brch$ such that \[\chi=\brch^{1-c}\] (\cf\cite[Lemma 5.31]{HidaDeepBlue}). Replacing $\brch$ by $\brch\cdot \xi\circ\rmN_{K/\Q}$ for a suitable Dirichlet character $\xi$, we may further assume $\brch$ satisfies the following minimal condition
\beqcd{min}\text{the conductor of $\brch$ is minimal among Dirichlet twists.}\eeqcd  
Since $\chi=\brch^{1-c}$ is unramified outside $pd_K$, this in particular implies that the conductor $\frakc$ of $\brch$ has a decomposition \[\frakc=\frakc_{\rm i}\frakc_{\rm s},\quad (\frakc,pd_K)=1;\,(\ol{\frakc_{\rm s}},\frakc_{\rm s})=1,\] where $\frakc_{\rm i}$ is only divisible by primes inert in $K$ and $\frakc_{\rm s}$ is only divisible by primes split in $K$. The level of the associated CM form $\theta_\brch^\circ$ is $N=d_KC_{\rm i}^2C_{\rm s}$, where $C_{\rm i}$ and $C_{\rm s}$ are positive integers satisfying $(C_{\rm i})=\frakc_{\rm i}\cap \Z$ and $(C_{\rm s})=\frakc_{\rm s}\cap \Z$. Put 
 \[C=d_K C_{\rm i}C_{\rm s}.\]
With the transfer map $\sV:\Lam\to \Lam_K$ in \eqref{E:22}, we define
 \beq\label{E:31}\cG_C:=\sV\left(\frac{2C}{\mathcal{L}^{\rm KL}_p(\bfone)}\cdot \cE_C\right)\in\Lam_P\powerseries{q},\eeq
where $\Lam_P$ is the localization of $\Lam_K$ at $P$.
By construction and the fact that $\zeta_p(s)$ has a simple pole at $s=1$, we find that \beq\label{E:310}\cG_C\con [\Frob_C^{-1}]\con 1\pmod{P}.\eeq Let $e_\Ord$ be Hida's ordinary projector on the space of $\Lam$-adic forms. The spectral decomposition of  \[e_\Ord(\theta_\brch^\circ\cG_C)\in \bfS=\bfS^\Ord(N,\brch_+\tau_{K/\Q},\Lam_K)\] according to \eqref{E:decomp.1} allows us to make the following definition.
\begin{defn}\label{D:H.3}Let $\aA$ and $\bB$ be the unique elements in $\bfK$ such that
 \beq\label{E:32}\sH:=e_\Ord(\theta_\brch^\circ \cG_C)-\aA\cdot \bftheta_\brch-\bB\cdot\bftheta_{\brch^c}\in \bfS^\perp.\eeq\end{defn}
 Let $\bfc$ be the positive integer such that $\bfc\cO_K$ is the conductor of $\chi$.
 \begin{prop}\label{P:3RS}With the ray class character $\brch$ and the integer $C$ as above, we have 
 \[\e_\frakp^s(\bB)=\frac{2\cL_p(s,0,\bfone)\cL_p(s,0,\chi)}{L(0,\tau_{K/\Q})\cL_p(s,-s,\chi)\zeta_p(1-s)}\cdot\frac{\Dmd{d_K\bfc}^s}{(1-\e_{\frakp}^s(\Frob_{\pbar}))^2}  . \]
  \end{prop}
 \begin{proof} 
  This can be proved by Hida's $p$-adic Rankin-Selberg method. We shall use the representation theoretic approach in \cite{HsiehChen20}. We follow the notation in \cite[Section 5, Section 6]{HsiehChen20}.  It suffices to show that for all but finitely many positive integer $k$ with $k\con 0\pmod{p-1}$, 
  \beq\label{E:34}\e_\frakp^k(\bB)=\frac{2\cL_p(k,0,\bfone)\cL_p(k,0,\chi)}{L(0,\tau_{K/\Q})\cL_p(k,-k,\chi)\zeta_p(1-k)}\cdot\frac{(d_K\bfc)^k}{(1-\psi^k(\Frob_\pbar))^2}. \eeq
 Here recall that $\psi$ is the idele class character of $K$ corresponding to $\e_\frakp$ with $\psi_\infty(z)=z$. To evaluate $\e_\frakp^k(\bB)$, we consider the spectral decomposition \beq\label{E:36}\begin{aligned}&\frac{2C}{\zeta_p(1-k)}\cdot e_\Ord(\theta_\brch^\circ E_k(C))\\
 =&\cC_k(\brch,\brch^c)\cdot\theta_{\brch\psi^{-k}}^{(\frakp)}+\cC_k(\brch^c,\brch)\cdot \theta_{\brch^c\psi^{-k}}^{(\frakp)}+\sH_{k}\in \cS_{k+1}(Np,\brch_+^{-1}\tau_{K/\Q}),\end{aligned}\eeq
  where $\sH_{k}$ is orthogonal to the space spanned by $\theta_{\brch^{-1}\psi^{-k}}, \theta_{\brch^{-1}\psi^{-k}}^{(\frakp)}$, $\theta_{\brch^{-c}\psi^{-k}}$ and $\theta_{\brch^{-c}\psi^{-k}}^{(\frakp)}$ under the Petersson inner product. Since $\e_\frakp^k(\bftheta_\brch)=\theta_{\brch\psi^{-k}}^{(\frakp)}$ is a $p$-stabilized newform of weight $k+1$, the decomposition \eqref{E:36} is indeed obtained by the image of \eqref{E:32} under the map $\e_\frakp^k$, and hence \[\e_\frakp^k(\bB)=\iota_p^{-1}(\cC_k(\brch^c,\brch)).\] 
 Now we use the adelic Rankin-Selberg method to compute the value $\cC_k(\brch^c,\brch)$. Let $f^\circ=\theta_{\brch^{-1}\psi^{-k}}$ and $g^\circ=\theta_\brch^\circ$ be the newforms associated with Hecke characters $\brch^{-1}\psi^{-1}$ and $\brch$. Let $\om:=\brch_+^{-1}\tau_{K/\Q}^{-1}$ viewed as an idele class character of $\Q$. Let $\varphi_{f^\circ}:=\varPhi(f^\circ)\in \cA_{k+1}(\om)$ and $\varphi_{g^\circ}=\varPhi(g^\circ)\in\cA_1(\om^{-1})$ be the automorphic newforms corresponding to $f^\circ$ and $g^\circ$ via the map $\varPhi$ in \cite[(2.4)]{HsiehChen20}. Let $\pi_1$ and $\pi_2$ be the \emph{unitary} cuspidal automorphic representation of $\GL_2(\A)$ associated with $\varphi_{f^\circ}$ and $\varphi_{g^\circ}$. The  $\pi_1$ and $\pi_2$ are the automorphic inductions of the idele class characters $\brch^{-1}\psi^{-k}\Abs_{\A_K}^{\frac{k}{2}}$ and $\brch$, and the automorphic forms $\varphi_{f^\circ}$ and $\varphi_{g^\circ}$ are normalized new vectors in $\pi_1$ and $\pi_2$. In addition, we have the equality of automorphic $L$-functions and Dirichlet series of modular forms
 \begin{align*}L(s,\pi_1)=&\Gamma_\C(s+\frac{k}{2})D(s+\frac{k}{2},f^\circ)=L(s+\frac{k}{2},\brch^{-1}\psi^{-k});\\
 L(s,\pi_2)=&\Gamma_\C(s)L(s,g^\circ)=L(s,\brch).\end{align*}
 Let $f=\theta_{\brch^{-1}\psi^{-k}}^{(\frakp)}$ be the $\frakp$-stablized newform associated with $f^\circ$ and let $\breve f:=\theta_{\brch^c\psi^{-k}}^{(\frakp)}$ be the specialization of $\e_\frakp^s(\bftheta_{\brch^c})$ at $s=k$. Then the automorphic representation generated by the associated automorphic forms $\varphi_{\breve f}$ is the contragredient representation $\pi_1^\vee=\pi_1\ot\om^{-1}$. Define the $\C$-linear pairing $\pairing:\cA^0_{-k-1}(\om)\times\cA_{k+1}(\om^{-1})\to\C$ by 
  \[\pair{\varphi_1}{\varphi_2}=\int\limits_{\A_\Q^\x\GL_2(\Q)\bksl \GL_2(\A_\Q)}\varphi_1(g)\varphi_2(g)\rmd^{\rmt} g.\]
Here $\rmd^\rmt g$ is the Tamagawa measure of $\PGL_2(\A)$. By \cite[Proposition 5.2]{HsiehChen20}, for $n\gg 0$ large enough, we have
 \begin{align*}\cC_k(\brch^c,\brch)=&\frac{\pair{\rho(\cJ_\infty t_n)\varphi_f}{\varphi_{g^\circ}\cdot E_\A(-,f_{\cD,s-1/2})}|_{s=1-\frac{k}{2}}}{\pair{\rho(\cJ_\infty t_n)\varphi_f}{\varphi_{\breve f}}}\cdot \frac{2C}{\zeta_p(1-k)},
 \end{align*}
 where $\cJ_\infty=\pDII{-1}{1
 }\in\GL_2(\R)$ and $t_n=\pMX{0}{p^{-n}}{-p^n}{0}\in\GL_2(\Qp)$. In order to explain the calculation of $\cC_k(\brch^c,\brch)$ by the adelic Rankin-Selberg method, we need to prepare some notation from the theory of automorphic representations. For any cuspidal automorphic representation $\pi$ of $\GL_2(\A)$, let $\cW(\pi)$ denote the Whittaker model of $\pi$ associated with the additive character $\bfe:\A/\Q\to\C^\x$. For each place $v$ of $\Q$, let $\cW_v(\pi)$ be the local component of $\cW(\pi)$ at $v$. For $(W_1,W_2)\in \cW_v(\pi_1)\times\cW_v(\pi_2)$, let $\Psi(W_1,W_2,f_{\cD,s,v})$ be the local zeta integral defined in \cite[(5.10)]{HsiehChen20}.
 If $v$ is finite, let $W_{\pi,v}\in \cW_v(\pi)$ be the new Whittaker function with $W_{\pi,v}(1)=1$ and if $v=\infty$ and $\pi_\infty$ is discrete series, let $W_{\pi,v}$ be the Whittaker of minimal $\SO(2)$-type with $W_{\pi,\infty}(1)=1$ (\cf\cite[\S 2.6.4]{HsiehChen20}).  For $\varphi\in \cA_0(\om)$, the Whittaker function $W_\varphi:\GL_2(\A)\to\C$ is defined by  
 \[W_{\varphi}(g)=\int_{\A/\Q}\varphi(\pMX{1}{x}{0}{1}g)\bfe(-x)\rmd x.\] 
In our setting, the Whittaker functions of $\varphi_f\in\pi_1$ and $\varphi_{g^\circ}\in \pi_2$ are given by
\[W_{\varphi_f}=W_{\pi_1,p}^\Ord\prod_{v\neq p}W_{\pi_1,v};\quad W_{\varphi_{g^\circ}}=\prod_v W_{\pi_2,v},\]
where $W_{\pi_1,p}^\Ord\in \cW(\pi_{1,p})$ is the ordinary Whittaker function characterized by $W_{\pi_1,p}^\Ord(\pDII{a}{1})=\alpha_f(a)\abs{a}_{\Qp}^\onehalf\bbI_{\Zp}(a)$, where $\al_f:\Qp^\x\to\C^\x$ is the unramified character with $\al_f(p)=\brch^{-1}\psi^{-k}(\pbar)p^{-\frac{k}{2}}$ (See \cite[Definition 2.1]{HsiehChen20}). Following \cite[Chapter V]{Jacquet72LNM278} (\cf\cite[(5.11)]{HsiehChen20}), we have the identity\begin{align*}&\pair{\rho(\cJ_\infty t_n)\varphi_f}{\varphi_{g^\circ}\cdot E_\A(-,f_{\cD,s})}=\int\limits_{\PGL_2(\Q)\bksl \PGL_2(\A)}\varphi_f(g\cJ_\infty t_n)\varphi_{g^\circ}(g)E_\A(g,f_{\cD,s})\rmd^{\rmt}g\\
&=\frac{1}{\zeta_\Q(2)}\Psi(W_{\pi_1,p}^\Ord,W_{\pi_2,p},f_{\cD,s,p})\Psi(\rho(\cJ_\infty)W_{\pi_1,\infty},W_{\pi_2,\infty},f_{\cD,s,\infty})\prod_{v\neq p,\infty}\Psi(W_{\pi_1,v},W_{\pi_2,v},f_{\cD,s,v}).
\end{align*} By the calculation in \cite[Proposition 5.3]{HsiehChen20} with $k_1=k_3=k+1$ and $k_2=1$, we find that
 \beq\label{E:33}\begin{aligned}&\pair{\rho(\cJ_\infty t_n)\varphi_f}{\varphi_{g^\circ}\cdot E_\A(-,f_{\cD,s-1/2})}|_{s=1-\frac{k}{2}}\\=&\frac{L(s,\pi_1\times\pi_2)}{\zeta_\Q(2)[\SL_2(\Z):\Gamma_0(N)]}\cdot\frac{(\sqrt{-1})^k}{2^{k+2}}\cdot\Psi_p(s)\prod_{\ell\mid N}\Psi^*_\ell(s)\Big\vert_{s=1-\frac{k}{2}},\end{aligned}\eeq
where $L(s,\pi_1\times\pi_2)$ is the Rankin-Selberg $L$-function for $\pi_1\times\pi_2$, $\Psi^*_\ell(s)$ and $\Psi_p(s)$ are local zeta integrals defined by \begin{align*}\Psi^*_\ell(s)&=\frac{\zeta_{\Q_\ell}(1)}{\zeta_{\Q_\ell}(2)\abs{N}_{\Q_\ell}}\frac{\Psi(W_{\pi_{1,\ell}},W_{\pi_{2,\ell}},f_{\Phi_{\cD,\ell},s-1/2})}{L(s,\pi_{1,\ell}\times\pi_{2,\ell})}\text{ if }\ell\neq p,\\
\Psi_p(s)&=\frac{\Psi ( \rho (t_n)W_{\pi_{1,p}}^{\rm ord},W_{\pi_2,p},f_{\Phi_{\cD,p},s-1/2})}{L(s,\pi_{1,p}\times\pi_{2,p})}.\end{align*}
Note that $N$ is the conductor of $\pi_1$ and $\pi_2$. Let $\supp(N)$ be the set of prime divisors of $N$. In \cite[\S 5.1, page 220]{HsiehChen20}, to $(\pi_1,\pi_2)$, we associate a decomposition $\supp(N)=\Sigma_{(\rm i)}\disjoint\Sigma_{(\rm ii)}\disjoint\Sigma_{(\rm iii)}$, and in our case, $\ell\in\Sigma_{(\rm i)}$ if $\ell\mid d_KC_{\rm s}$, $\ell\in \Sigma_{(\rm ii)}$ if $\ell\mid C_{\rm i}$ and $\Sigma_{(\rm iii)}=\emptyset$. According to the computation of local zeta integrals $\Psi^*_\ell(s)$ in \cite[Lemma 6.3, Lemma 6.5]{HsiehChen20} at $\ell\mid N$, we have
\[\Psi_\ell^*(s)=1\text{ if }\ell\mid d_KC_{\rm s};\quad \Psi^*_\ell(s)=\abs{C_{\rm i}}_{\Q_\ell}^{-1}(1+\ell^{-1})\text{ if }\ell\mid C_{\rm i}.\]
We compute the local zeta integral $\Psi_p(s)$ by a similar calculation in \cite[Lemma 6.1]{HsiehChen20}. Put $W_1=W_{\pi_1,p}^\Ord$ and $W_2=W_{\pi_2,p}$. Then $\Psi ( \rho (t_n)W_{\pi_{1,p}}^{\rm ord},W_{\pi_2,p},f_{\Phi_{\cD,p},s-1/2})$ equals 
\begin{align*}
&\frac{\zeta_{\Qp}(2)}{\zeta_{\Qp}(1)}\int_{\Qp^\x}\int_{\Qp}W_1(\pDII{y}{1}\pMX{0}{-1}{1}{x}t_n)W_2(\pDII{-y}{1}\pMX{0}{-1}{1}{x})
\abs{y}_{\Qp}^{s-1}\\
&\times f_{\Phi_{\cD,p},s-\onehalf}(\pMX{0}{-1}{1}{x})\rmd x\rmd^\x y\\
=&\frac{\zeta_{\Qp}(2)}{\zeta_{\Qp}(1)}\int_{\Qp}\int_{\Qp^\x}W_1(\pDII{yp^n}{p^{-n}}\pMX{1}{0}{-p^{2n}x}{1})W_2(\pDII{-y}{1}\pMX{0}{-1}{1}{x})\abs{y}_{\Qp}^{s-1}\\
&\times \bbI_{\Zp}(x)\rmd^\x y\rmd x\\
=&\frac{\zeta_{\Qp}(2)\al_f\Abs_{\Qp}^\onehalf(p^{2n})\om_{p}^{-1}(p^n)}{\zeta_{\Qp}(1)}\int_{\Qp^\x}W_2(\pDII{-y}{1})\al_f\Abs_{\Qp}^{s-\onehalf}(y)\rmd^\x y\\
=&\frac{\om_{p}^{-1}\al_f^2\Abs_{\Qp}(p^{n})\zeta_{\Qp}(2)}{\zeta_{\Qp}(1)}\cdot L(s,\pi_{2,p}\ot\al_f),
\end{align*}
so we obtain that
\[\Psi_p(s)=\frac{\om_{p}^{-1}\al_f^2\Abs_{\Qp}(p^{n})\zeta_{\Qp}(2)}{\zeta_{\Qp}(1)}\cdot L(s,\pi_{1,p}\times\pi_{2,p})(1-\psi^{-k}(\frakp)p^{-s-\frac{k}{2}})(1-\brch^{c-1}\psi^{-k}(\frakp)p^{-s-\frac{k}{2}}).\]
From the above equations with the equality of $L$-functions\[L(s,\pi_1\times\pi_2)=L(s+k/2,\theta_\brch^\circ\ot\theta_{\brch^{-1}\psi^{-k}})=L(s+k/2,\psi^{-k})L(s+k/2,\brch^{c-1}\psi^{-k}), \] we find that \eqref{E:33} equals
 \begin{align*}
&\pair{\rho(\cJ_\infty t_n)\varphi_f}{\varphi_{g^\circ}\cdot E_\A(-,f_{\cD,s-1/2})}|_{s=1-\frac{k}{2}}=\frac{L(1,\psi^{-k})L(1,\brch^{c-1}\psi^{-k})}{\zeta_\Q(2)[\SL_2(\Z):\Gamma_0(N)]}\cdot  \frac{(\sqrt{-1})^k}{2^{k+2}}\\
&\times
  (1-\psi^{-k}(\frakp)p^{-1})(1-\brch^{c-1}\psi^{-k}(\frakp)p^{-1})\cdot \frac{\om_{p}^{-1}\al_f^2\Abs_{\Qp}(p^{n})\zeta_{\Qp}(2)}{\zeta_{\Qp}(1)}\cdot
  C_{\rm i}\prod_{q\mid C_{\rm i}}(1+q^{-1}).
 \end{align*}
On the other hand, by \cite[Lemma 3.6]{Hsieh21AJM}, 
 \[\pair{\rho(\cJ_\infty t_n)\varphi_f}{\varphi_{\breve{f}}}=\frac{\norm{f^\circ}_{\Gamma_0(N)}^2\cE(f,\Ad)}{\zeta_\Q(2)[\SL_2(\Z):\Gamma_0(N)]}\cdot\frac{\om_{p}^{-1}\al_f^2\Abs_{\Qp}(p^n)\zeta_{\Qp}(2)}{\zeta_{\Qp}(1)},\]
 where $\cE(f,\Ad)=(1-\chi\psi^{(1-c)k}(\frakp^{-1})p^{-1})(1-\chi\psi^{(1-c)k}(\pbar))
$. By the minimal condition \eqref{min} of $\brch$, the level of the newform $f^\circ=\theta_\brch^\circ$ is minimal among its Dirichlet twists. By \cite[Theorem 7.1]{HidaTil93ASENS}, 
 \[\norm{f^\circ}_{\Gamma_0(N)}^2=2^{-(k+1)}L(1,\pi_1,\Ad)\cdot N\prod_{\ell\mid C_{\rm i}}(1+\ell^{-1})\quad (N=CC_{\rm i}).\]
Put $\psi_-=\psi^{1-c}$. From the above equations, we deduce that
 \beq\label{E:37}\cC_k(\brch^c,\brch)=\frac{(\sqrt{-1})^kL(1,\psi^{-k})L(1,\brch^{c-1}\psi^{-k})}{2L(1,\pi_1,\Ad)}\cdot \frac{ (1-\psi^{-k}(\frakp)p^{-1})(1-\brch^{c-1}\psi^{-k}(\frakp)p^{-1})}{(1-\chi\psi_-^k(\frakp)^{-1}p^{-1})(1-\chi\psi_-^k(\pbar))} \cdot \frac{2}{\zeta_p(1-k)}.\eeq
 By the functional equations of $L$-functions, one has
 \begin{align*}
 L(1,\psi^{-k})L(1,\brch^{c-1}\psi^{-k})&=\varepsilon(1,\psi^{-k})\varepsilon(1,\chi^{-1}\psi^{-k})L(0,\psi^k)L(0,\chi\psi^k),\\
 L(1,\pi_1,\Ad)&=L(1,\tau_{K/\Q})L(1,\brch^{1-c}\psi^{(1-c)k})\\
 &=\sqrt{d_K}^{-1}\varepsilon(1,\chi\psi_-^k)\cdot L(0,\tau_{K/\Q})L(0,\chi\psi_-^k).
 \end{align*}
 Since $\psi^k$ is unramified everywhere and $\psi_-|_{\A^\x_\Q}=1$, we have \begin{align*}
 \varepsilon(1,\psi^{-k})&=\sqrt{d_K}^{-1}(\sqrt{-d_K})^k,\quad
  \varepsilon(1,\chi\psi^{-k})=\varepsilon(1,\chi)(\sqrt{-d_K}\bfc)^k,\\
\varepsilon(1,\chi\psi_-^k)& =\varepsilon(1,\chi).\end{align*}
It follows that
\beq\label{E:38}\frac{(\sqrt{-1})^kL(1,\psi^{-k})L(1,\brch^{c-1}\psi^{-k})}{L(1,\pi_1,\Ad)}=\frac{L(0,\psi^k)L(0,\chi\psi^k)}{(\sqrt{-1})^kL(0,\tau_{K/\Q})L(0,\chi\psi_-^k)}\cdot (d_K\bfc)^k.\eeq
 By the interpolation formulae of the Katz $p$-adic $L$-function in \propref{P:3interpolation}, we find that
\beq\label{E:39}\begin{aligned}\frac{\cL_p(k,0,\chi)\cL_p(k,0,\bfone)}{\cL_p(k,-k,\chi)}=&\frac{1}{2(\sqrt{-1})^k}\frac{L(0,\chi\psi^{k})L(0,\psi^{k})}{L(0,\chi\psi_-^k) }\\
&\times \frac{ (1-\psi^{k}(\frakp^{-1})p^{-1})(1-\chi\psi^{k}(\frakp^{-1})p^{-1})(1-\psi^k(\pbar))^2}{(1-\chi\psi_-^k(\frakp^{-1})p^{-1})(1-\chi\psi_-^k(\pbar))}.\end{aligned}\eeq
Combining \eqref{E:37}, \eqref{E:38} and \eqref{E:39}, we obtain \eqref{E:34}.   \end{proof}
Recall that $\Lam_P$ is the localization of $\Lam_K$ at the augmentation ideal $P$. Let $h$ be the class number of $K$ and let $\uf\in K^\x$ be a generator of $\frakp^h$. Put
 \[\sL(\bfone):=-\frac{\log_p\uf}{h}=\frac{\log_p\ol{\uf}}{h}\neq 0.\]
\begin{cor}\label{C:3RS}We have $\bB^{-1}\in P\Lam_P$. Let $B(s):=\e_\frakp^s(\bB^{-1})$. Then  
 \[
 \frac{d}{ds}B(s)\Big\vert_{s=0}=\frac{\cL'_p(s,-s,\chi)\rvert_{s=0}}{\cL_\frakp^*(0,\chi)}.\]
 \end{cor}
 \begin{proof}
By \propref{P:3RS},
\beq\label{E:310} B(s)=\cL_p(s,-s,\chi)\cdot \frac{L(0,\tau_{K/\Q})}{2\cL_p(s,0,\bfone)\cL_\frakp^*(s,\chi)}\cdot\frac{\zeta_p(1-s)(1-\e_{\frakp}^s(\Frob_{\pbar}))}{\Dmd{d_K\bfc}^s}. \eeq
By the residue formula of the $p$-adic zeta function,\[
 \zeta_p(1-s)(1-\e_{\frakp}^s(\Frob_{\pbar}))|_{s=0}=(p^{-1}-1)\sL(\bfone)\neq 0.\]
 On the other hand, from Katz's $p$-adic Kronecker limit formula \cite[Theorem 5.2, page 88]{Shalit87book} and the fact that $L(0,\tau_{K/\Q})=2h/\#(\cO_K^\x)$, we deduce that
 \[ \cL_p(0,0,\bfone)=(1-p^{-1})\frac{\log_p \ol{\uf}^{-1}}{\#(\cO_K^\x)}=(p^{-1}-1)\sL(\bfone)\cdot 2^{-1}L(0,\tau_{K/\Q})\neq 0.\]
By \eqref{E:improved2} and the Brumer-Baker theorem, $\cL_\frakp^*(0,\chi)\neq 0$ and $B(0)=0$. We thus conclude from \eqref{E:310} that $\bB^{-1}\in P\Lam_P$ and the desired formula of the derivative $B'(0)$.
 \end{proof}
 \section{Galois cohomology classes and $\sL$-invariants}
 \subsection{Cohomological interpretation of $\sL$-invariants}
Let $F=\Frac\cW$. As in the previous section, $\chi:\Gal(K(\bfc)/K)\to F^\x$ is a non-trivial ring class character unramified outside $pd_K$ with $\chi(\pbar)=1$, and $\brch$ is a ray class character of conductor $\frakc$ with $\chi=\brch^{1-c}$. For a finite set $S$ of primes of $\cO_K$, denote by $\rmH^1_S(K,\chi)$ the subspace of cohomology classes unramified outside $S$. By the global Poitu-Tate duality, it is known that $\rmH^1_{\emptyset}(K,\chi)=\rmH^1_{\stt{\pbar}}(K,\chi)=\stt{0}$ and $\dim_F\rmH^1(K,\chi)=\dim_F\rmH^1_{\stt{\frakp,\pbar}}(K,\chi)=1$ (\cf \cite[Proposition 1.3]{BeD21Adv}). Let $\loc_{\pbar}:\rmH^1(K,\chi)\to \rmH^1(K_{\pbar},\chi)=\Hom(G_{K_\pbar},F)$ be the localization at $\pbar$. With the embedding $\iota_p:K\hookrightarrow\Qp$, we identity $K_p:=K\ot\Qp\iso \Qp\oplus \Qp$ by $\al\ot x\mapsto (\iota_p(\al) x,\iota_p(\ol{\al})x)$. Let \[\rec_{K_p}:K_p^\x=\Qp^\x\oplus\Qp^\x\to C_K\stackrel{\rec_K}\longto G_K^{ab}\] be the composition of the natural inclusion $K_p^\x\hookto C_K=K^\x\bksl \A_K^\x$ and the reciprocity law map $\rec_K$. Therefore, for any $\kappa\in \rmH^1(K,\chi)$, we can identify $\loc_\pbar(\kappa)\in\Hom(G_{K_\pbar},\Cp)$ with an element in $\Hom(\Qp^\x,\Cp)$ by \[\loc_\pbar(\kappa)(a)=\kappa(\rec_{K_p}(1,a))\text{ for }a\in\Qp^\x.\]
  \begin{lm}\label{L:41} Let $\kappa$ be a non-zero class in $\rmH^1(K,\chi)$ and write
 \[\loc_{\pbar}(\kappa)=x \cdot \Ord_p+y\cdot  \log_p.\] Then $y\neq 0$, and  \[\sL(\chi)=\frac{x}{y}.\]
 \end{lm}
 \begin{proof}
 First we note that $y\neq 0$ since $\rmH^1_{\stt{\pbar}}(K,\chi)=0$. Let $\loc_\frakp(\kappa)=w\cdot \Ord_p+z\cdot \log_p$. By the relation $\pair{\loc_\frakp(\kappa)}{\loc_\frakp(x)}+\pair{\loc_{\pbar}(\kappa)}{\loc_\pbar(x)}=0$ for $x=\ux$ or $\vx$, we obtain the equations
\begin{align*}
z\cdot \log_\frakp \ux&+y\cdot \log_{\pbar}(\ux)=0;\\
z\cdot \log_\frakp \vx&+x\cdot \Ord_{\ol{\frakP}}(\vx)+y\cdot \log_{\pbar}\vx=0.
\end{align*}
The lemma now follows. 
\end{proof}

 \subsection{Construction of cohomology classes}
Let $S$ be the set of prime factors of $\frakp\bfc$. Let $G_{K,S}=\Gal(K_S/K)$, where $K_{S}$ is the maximal algebraic extension of $K$ unramified outside $S$. Let $\bfT^\perp\subset \End\bfS^{\perp}$ be the image of the Hecke algebra $\bfT=\bfT(N,\brch_+\tau_{K/\Q},\Lam_K)$ restricted to $\bfS^\perp$. Then $\bfT^\perp$ is a finite flat $\Lam_K$-algebra. Fix a generator $\gamma_0$ of $\Gamma_{K,\frakp}$. Then $\Lam_K$ can be identified with $\cW\powerseries{X}$ and the augmentation ideal $P$ of $\Lam_K$ is the principal ideal generated by $X=\gamma_0-1$. We use the argument in \cite[Theorem 4.2]{DDP11Ann} to construct nonzero cohomology classes in the following
 \begin{thm}\label{T:Ribet}
 Let $\lam:\bfT^\perp\to \Lam_P/(X^{n+2})$ be a $\Lam_K$-algebra homomorphism. Let $\boldsymbol\alpha:G_{K_\pbar}\to\Lam_P^\x$ be the unique unramified character such that $\boldsymbol\alpha(\Frob_{\pbar})=\lam(U_p)$. Suppose that there exists a character $\wtd\Psi:G_{K,S}\to \Lam_P/(X^{n+2})$ such that \begin{itemize}\item[(i)] $\wtd\Psi\con 1\pmod{X}$, \item[(ii)]
 $\lam(T_\ell)=\brch\wtd\Psi(\Frob_\frakl)+\brch\wtd\Psi(\Frob_{\ol{\frakl}})$ for all $\ell\ndivides p$ splits in $K$ and \item[(iii)]$\wtd\Psi|_{G_{K_\pbar}}\con\brch^{-1}{\boldsymbol \alpha}-\eta X^{n+1}\pmod{X^{n+2}}$ for some non-zero homomorphism $\eta:G_{K_\pbar}\to\Cp$.\end{itemize}  Then there exists $\kappa\neq 0\in \rmH^1(K, \chi)$ such that
\[\loc_{\pbar}(\kappa)=\frac{\brch^{-1}{\boldsymbol\alpha}-\wtd\Psi|_{G_{K_\pbar}}}{X^{n+1}}\Big\vert_{X=0}=\eta.\]
 \end{thm}
 \begin{proof}
 \def\Tdag{\bfT^\dagger}
Let $\Lam^\dagger\supset\Lam_P$ be the local ring of rigid analytic functions around $X=0$, \ie
\[\Lam^\dagger=\stt{\sum_{n=0}^\infty a_n X^n\in F\powerseries{X} \, \middle| \text{ there exist $r>0$ such that }\lim_{n\to\infty}\abs{a_n}r^n=0}.\]
 Let $\Tdag=\bfT^\perp\ot_{\Lam_K}\Lam^\dagger$ be a finite $\Lam^\dagger$-algebra, and hence a finite product of henselian local rings. Let $I$ be the kernel of the map $\lam:\Tdag\to \Lam^\dagger/(X^{n+2})$.   Let $\sT:G_{\Q,S}\to \bfT^\perp\to \Tdag$ be the pseudo character defined by $\sT(\Frob_\ell)=T_\ell$. The assumption (ii) implies that \[\sT|_{G_{K,S}}\con \phi\wtd\Psi+\phi^c\wtd\Psi^c\pmod{I}.\] Since $\brch\neq\brch^c$, applying the theory of residually multiplicity free pseudo characters \cite[Theorem 1.4.4]{BC09Ast324}) to $\sT|_{G_{K,S}}$, we obtain a continuous representation $\rho_\lam:G_{K,S}\to \GL_2(\Frac \Tdag)$ such that 
the image of $\rho_\lam(\Tdag[G_{K,S}])$ is a generalized matrix algebra of the form
\[\rho_\lam(\Tdag[G_{K,S}])=\pMX{\Tdag}{\frakt_{12}}{\frakt_{21}}{\Tdag},\]
where $\frakt_{ij}$ are fractional $\Tdag$-ideals in $\Frac \Tdag$ and $\frakt_{12}\frakt_{21}\subset I$. Writing 
 \[\rho_\lam(\sg)=\pMX{a(\sg)}{b(\sg)}{c(\sg)}{d(\sg)}\text{ for }\sg\in G_{K,S},\] then we have \[a(\sg)\con \brch\wtd\Psi(\sg)\pmod{I};\quad d(\sg)\con\brch\wtd\Psi(c\sg c)\pmod{I}.\] 
 Note that $\Tdag/I\iso \Lam^\dagger/(X^{n+2})$ is a local ring. Let $Q$ be the maximal ideal of $\Tdag$ containing $I$. Let $R=\Tdag_Q$ be the localization of $\Tdag$ at $Q$. Then $R$ is a finite flat and reduced $\Lam^\dagger$-algebra since $N$ is the tame conductor of $\theta_\brch^\circ$. Put $R_{ij}:=\frakt_{ij}\ot_{\Tdag} R$. By \cite[Theorem 6.12]{HidaTil94Inv}, there exists $\pMX{A}{B}{C}{D}\in\GL_2(\Frac R)$ such that 
 \[\pMX{a(\sg)}{b(\sg)}{c(\sg)}{d(\sg)}\pMX{A}{B}{C}{D}=\pMX{A}{B}{C}{D}\pMX{\boldsymbol\alpha(\sg)}{*}{0}{*} \text{ for all }\sg\in G_{K_\pbar},\]
 and hence
 \beq\label{E:41}C\cdot b(\sg)=A\cdot (\boldsymbol\alpha(\sg)-a(\sg))\text{ for }\sg\in G_{K_\pbar}.\eeq
We claim that $R_{12}$ is a faithful $R$-module. To see the claim, we note that by the reducedness of $R$, the fractional field $\Frac (R)$ is isomorphic to a product of fields \[\Frac(R)=\prod_{i=1}^t L_{\sH_i},\] and each field $L_{\sH_i}$ is a finite extension of $\Lam^\dagger$ and corresponds to a cuspidal Hida family $\sH_i$. For $i=1,\dots,t$, let $\pi_i:\Frac(R)\to L_{\sH_i}$ be the natural projection map. Then $\rho_{\sH_i}:=\rho_\lam\circ\pi_i:G_{K,S}\to \GL_2(L_{\sH_i})$ is the Galois representation associated with $\sH_i$. If $R_{12}$ is not faithful, then $\pi_{i_0}(R_{12})=0$ for some $i_0$, and hence $\rho_{\sH_{i_0}}$ is reducible, which implies that $\sH_{i_0}$ is the Hida family $\Theta_{\brch_1}$ of CM forms associated with some ray class character $\brch_1\neq \brch$ or $\brch^c$ whose specialization at some arithmetic point $P'$ above $P$ agree with $\theta_\brch^{(\frakp)}$, which in turns suggests that $\brch_1+\brch_1^c=\brch+\brch^c$, and $\brch_1=\brch$ or $\brch^c$, a contradiction to the choice of $R$. 
 
Define the function $\sK:G_{K,S}\to R_{12}$ by $\sK(\sg)=b(\sg)/d(\sg)$. For any $R$-submodule $J\supset QR_{12}$ of $R_{12}$,  the reduction of $\sK$ modulo $J$  \[\ol{\sK}:=b/d\pmod{J}=\phi^{-c} b\pmod{J}:G_{K,S}\to R_{12}/J\] 
 is a continuous one-cocycle in $Z^1(G_{K,S},\chi\ot R_{12}/J)$. 
 We claim that if the class $[\ol{\sK}]\in \rmH^1(K,\chi\ot R_{12}/J)$ represented by $\ol{\sK}$ is zero, then $R_{12}=J$. We can write $b(\sg)\pmod{J}=(\phi^c(\sg)-\phi(\sg))z$ for some $z\in R_{12}/J$. Consider the $\rho_\lam(R[G_{K,S}])$-module $(R_{12}/J,R/Q)^{\rmt}$. Then the line $R(z,1)^{\rmt}\subset (R_{12}/J,R/Q)^{\rmt}$ is stable under the action of $\rho_\lam(R[G_{K,S}])$. On the other hand, $\pMX{0}{0}{0}{1}\in \rho_\lam (R[G_{K,S}])$, so we find that $(0,1)^{\rmt}\in R(z,1)^\rmt$. This implies $z=0$ and $b(\sg)\pmod{J}$ is zero. Since $R_{12}$ is the $R$-module generated by $\stt{b(\sg)}_{\sg\in G_{K,S}}$, we conclude $R_{12}=J$. In particular, this shows that $\sK\pmod{QR_{12}}$ represents a non-zero class $\kappa$ in $\rmH^1(K,\chi\ot R_{12}/Q)$. 
 
 Let $R'_{12}$ be the submodule of $R_{12}$ generated by $\stt{b_\sg}_{\sg\in G_{K_\pbar}}$ and let $J:=QR_{12}+R_{12}'$.  Then $\ol{\sK}:G_K\to R_{12}/J$ is a cocycle which is trivial at $\pbar$, and $[\ol{\sK}]\in\rmH_{\stt{\pbar}}^1(K,\chi)=\stt{0}$. By the above claim, we find that $J=R_{12}$ and hence $R'_{12}=R_{12}$ by Nakayama's lemma.  We next show that the element $C$ is invertible in $\Frac R$. Suppose not. Then $\pi_{i}(C)=0$ for some $\pi_i:\Frac(R)\to L_{\sH_i}$, and by \eqref{E:41} this would imply that $\boldsymbol\alpha(\sg)=a(\sg)$ for all $\sg\in G_{K_{\ol\frakp}}$, which contradicts to the assumption (iii).  Thus we have $C\in (\Frac R)^\x$ and  \[R_{12}=\frac{A}{C}X^{n+1}R.\] Since $\chi(\frakp)=1$, $\phi=\phi^c$ on $G_{K_\pbar}$, and it follows that for $\sg\in G_{K_\pbar}$, 
 \[\sK(\sg)\pmod{QR_{12}}=\frac{A}{C}X^{n+1}\eta(\sg)\pmod{\frac{A}{C} X^{n+2} R}=\eta(\sg)\pmod{QR_{12}}.\]
 Therefore, the non-zero class $\kappa=[\sK\pmod{QR_{12}}]\in \rmH^1(K,\chi)$ enjoys the required local description. This finishes the proof. 
 \end{proof}
Let $\Psi^{\rm univ}:G_{K,S}\to \Lam_K^\x$ be the universal character in \eqref{E:23}. By definition, $\Psi$ is unramified outside $\frakp$. For each $\sg\in G_{K,S}$, we can write $\Psi^{\rm univ}(\sg)\con 1+\eta_\frakp(\sg) X\pmod{X^2}$ for some $\eta_\frakp\in\Hom(G_{K,S}^{ab},\Cp)$. By definition, \beq\label{E:46}\e_\frakp(\Psi^{\rm univ}(\rec_{K_p}(u,1)))=\Dmd{u}^{-1}\text{ for }u\in \Zp^\x.\eeq  
Let $\bfv:=\e_\frakp(\gamma_0)$ and put
\[\eta_\pbar(\sg)=\eta_\frakp(c\sg c);\quad \eta^*_v=\log_p\bfv\cdot \eta_v,\,v=\frakp\text{ or }\pbar.\] 

\begin{lm}\label{L:43} We have
\[ \loc_{\pbar}(\eta^*_\frakp)=\sL(\bfone)\cdot \Ord_p  \quad \textup{and}\quad  \loc_{\pbar}(\eta^*_{\pbar})=-\log_p-\sL(\bfone)\cdot \Ord_p.\]
\end{lm}
\begin{proof}Write $\Psi_s(\sg):=\e_\frakp^s(\Psi^{\rm univ}(\sg))=\e_\frakp(\sg)^{-s}$, and then by definition we have \beq\label{E:47}\frac{d}{ds}\Psi_s(\sg)|_{s=0}=\eta_\frakp(\sg)\cdot \log_p\bfv=\eta^*_\frakp(\sg)\quad(\e_\frakp^s(X)=\bfv^s-1).\eeq Recall that $\sL(\bfone)=\frac{\log_p\ol{\uf}}{h}$, where $h$ is the class number of $K$ and $\uf\in K^\x$ with $\frakp^h=\uf\cO_K$. Evaluating  both sides of \eqref{E:47} for $\sg=\rec_{K_p}(1,\uf)$, we obtain \[h\cdot \eta^*_\frakp(\Frob_{\pbar})=\frac{d}{ds}\Psi_s(\rec_{K_p}(1,\uf))|_{s=0}=\frac{d}{ds}\Psi_s(\rec_{K_p}(\ol{\uf}^{-1},1))|_{s=0}=\log_p\ol{\uf}=h\cdot \sL(\bfone).\]
Since $\eta^*_\frakp$ is unramified at $\pbar$, \[\loc_\pbar(\eta^*_\frakp)=\eta^*_\frakp(\Frob_\pbar)\cdot \Ord_p=\sL(\bfone)\cdot \Ord_p.\] 

From \eqref{E:46} and \eqref{E:47}, we find that \[\eta^*_\pbar(\rec_{K_p}(1,a))=\eta^*_\frakp(\rec_{K_p}(a,1))=-\log_pa\text{ for } a\in \Zp^\x.\] On other hand, $\Psi_s(\rec_{K_p}(\uf,1))=\Psi_s(\rec_{K_p}(1,\ol{\uf}^{-1}))=1$, so $\eta^*_\pbar(\rec_{K_p}(1,\uf))=\eta^*_{\frakp}(\rec_{K_p}(\uf,1))=0$. These equations imply that \[\loc_{\pbar}(\eta^*_{\pbar})(a)=\eta^*_{\pbar}(\rec_{K_p}(1,a))=-\log_p(a)-\sL(\bfone)\cdot \Ord_p(a)\text{ for }a\in\Qp^\x.\qedhere\] \end{proof}

 \begin{thm}\label{T:43}We have the following formula for $\sL$-invariant:
 \[\sL(\chi)=2\sL(\bfone)-\frac{\cL_p'(s,-s,\chi)|_{s=0}}{\cL^*_\frakp(0,\chi)}.\]
 \end{thm}
 \begin{proof} 
Let $B(s)=\e_\frakp^s(\bB^{-1})$ be as in \corref{C:3RS}. 
 In view of \lmref{L:41}, \lmref{L:43} and the formula \corref{C:3RS}, it suffices to construct a nonzero element $\kappa\in\rmH^1(K,\chi)$ such that $\loc_{\pbar}(\kappa)$ is a nonzero multiple of 
 \beq\label{E:48}\loc_{\pbar}(\eta^*_\frakp)-\loc_{\pbar}(\eta^*_\pbar)-\frac{d}{ds}B(s)\Big\vert_{s=0}\cdot \Ord_p=\log_p+\left(2\sL(\bfone)-\frac{d}{ds}B(s)\Big\vert_{s=0}\right)\cdot \Ord_p.\eeq
We shall use \thmref{T:Ribet} and adapt the calculations in \cite[\S 3]{Ventullo15CMH} to construct such a class. Recall that the Fourier expansion \[\bftheta_\brch=\sum\limits_{n=1}^\infty \bfa(n,\bftheta_\brch)q^n\]  of the $\Lam$-adic CM form $\bftheta_\brch$ is given by 
\beq\label{E:45} \bfa(\ell,\bftheta_\brch)=\begin{cases}\brch\Psi^{\rm univ}(\Frob_\frakl)+\brch\Psi^{\rm univ}(\Frob_{\ol{\frakl}})&\text{ if }\ell\cO_K=\frakl\ol{\frakl}\text{ is split},\\
0&\text{ if }\ell\text{ is inert,}\\
 \brch\Psi^{\rm univ}(\Frob_\frakl)&\text{ if }\ell\mid d_KC_{\rm s} \textup{ and }\,\frakl\mid (d_K\frakc_{\rm s},\ell).\end{cases} \eeq
The CM form $\bftheta_\brch$ is a $\Lam$-adic newform of tame level $N=d_KC_{\rm s}C_{\rm i}^2$, so we have \beq\label{E:43} T_\ell\bftheta_\brch=\bfa(\ell,\bftheta_\brch)\bftheta_\brch\text{ if }\ell\ndivides N,\quad U_\ell\bftheta_\brch=\bfa(\ell,\bftheta_\brch)\bftheta_\brch\text{ if }\ell\mid N.
 \eeq
 Consider the $\Lam$-adic cusp form $\sH\in\bfS^\perp$ constructed in \eqref{E:32}:
\[\sH= -\frac{1}{A}\bftheta_\brch-\frac{1}{B}\bftheta_{\brch^c}+e_\Ord( \theta_\brch^\circ\cG_C), \textup{ where } A=\aA^{-1}\text{ and }B=\bB^{-1}.\]Put \[b_1:=\frac{B}{X}\Big\vert_{X=0}=\frac{1}{\log_p\bfv}\cdot\frac{d}{ds}B(s)\Big\vert_{s=0}.\] 
There are three cases.
 \subsubsection*{Case (i): $\Ord_P(A+B)=\Ord_P(A)$}
 Define
 \[\sH_1:=(-B)\cdot \sH\pmod{X^2}= \frac{B}{A}\cdot \bftheta_\brch+\bftheta_{\brch^c}-B\cdot e_\Ord(\cG_C \theta_\brch^\circ)\pmod{X^2}.\]
 Let $F=\cW[\frac{1}{p}]$. Put
 \[u_1=\frac{A}{A+B}\Big\vert_{X=0}\in F^\x.\]
 Then $\bfa(1,\sH_1)\con u_1^{-1}\pmod{X}$.  Define the additive homomorphism $\psi_1:G_{K,S}\to F$ by 
  \[\psi_1:=(1-u_1)\eta_\frakp +u_1\eta_\pbar\]
and the character $\Psi_1\colon G_{K,S}\to \Lam_P/(X^2)$ by 
 \[\Psi_1=1+\psi_1X\pmod{X^2}.\]
By \eqref{E:43}, $\cG_C\con 1\pmod{X}$ \eqref{E:310} and the equations
 \begin{align}\label{E:42}&X\bftheta_\brch\con X\bftheta_{\brch^c}\con X\theta_\brch^{(\frakp)}\pmod{X^2},\quad X\sH_1=u_1^{-1}X\theta_\brch^{(\frakp)}\pmod{X^2};\\
 \label{E:44}&U_p\theta_\brch^\circ=\brch(\frakp)\theta_\brch^\circ+\brch(\pbar)\theta_\brch^{(\frakp)}\quad (\brch(\frakp)=\brch(\pbar)),\end{align}
we verify that $\sH_1$ is an eigenform modulo $X^2$ with\begin{align*}
 T_\ell \sH_1=&(\brch\Psi_1(\Frob_\frakl)+\brch\Psi_1(\Frob_{\ol{\frakl}}))\sH_1\text{ if $\ell\ndivides pN$ and $\ell\cO_K=\frakl\ol{\frakl}$ is split},\\
U_p\sH_1=&\brch(\pbar)(1+X(\eta_\frakp(\Frob_{\pbar})-u_1 b_1)) \sH_1.
 \end{align*}
 Since $\sH\in \bfS^\perp$, this induces a homomorphism $\lam_{\sH}:\bfT^\perp\to \Lam_P/(X^2)$ defined by 
 \[ \lam_{\sH}(t):=\bfa(1,t\cdot \sH_1)/\bfa(1,\sH_1)=\bfa(1,t\cdot \sH)/\bfa(1,\sH)\pmod{X^2}\] with 
 \[\lam_{\sH}(T_\ell)=\brch\Psi_1(\Frob_\frakl)+\brch\Psi_1(\Frob_{\ol{\frakl}})\text{ if $\ell\cO_K=\frakl\ol{\frakl}$ is split},\,\,
\brch(\pbar)^{-1}\lam_{\sH}(U_p)=1+X(\eta_\frakp(\Frob_{\pbar})-u_1 b_1).\]
By \thmref{T:Ribet} with $\wtd\Psi=\Psi_1$ and $n=0$, we find that
 there exists a nonzero class $\kappa\in \rmH^1(K,\chi)$ with  \[\loc_{\pbar}(\kappa)=\loc_{\pbar}(\eta_\frakp)-u_1 b_1\cdot \Ord_p-\loc_{\pbar}(\psi_1)=u_1(\loc_{\pbar}(\eta_{\frakp})-\loc_{\pbar}(\eta_{\pbar})-b_1\cdot \Ord_p).\]
  \subsubsection*{Case (ii): $\Ord_P(A+B)> \Ord_P(B)=\Ord_P(A)$} In this case, $\sH_1$ is not an eigenform of the Hecke algebra $\bfT$ but a generalized eigenform. Put
 \[u_2=\frac{B}{A}\Big\vert_{X=0}\in F^\x.\]
  Define the additive homomorphism $\psi_2:G_{K,S}\to F$ by 
  \[\psi_2:=u_2\eta_\frakp+\eta_\pbar\]
  and define the character $\Psi_2:G_{K,S}\to\Lam_K/(X^2)$ by $\Psi_2=1+X\psi_2$. Using the relations \eqref{E:43}, \eqref{E:42}  and \eqref{E:44}, we find that the Hecke algebra $\bfT$ stabilizes the two-dimensional subspace spanned by $X\theta_\brch^{(\frakp)}$ and $\sH_1$. In addition, we have 
  \begin{align*}T_\ell\sH_1
  &=(\brch(\frakl)+\brch(\ol{\frakl}))\sH_1+(\brch(\frakl)\psi_2(\Frob_\frakl)+\brch(\ol{\frakl})\psi_2(\Frob_{\ol{\frakl}}))X\theta_\brch^{(\frakp)},\\
  T_\ell X\theta_\brch^{(\frakp)}&=(\brch(\frakl)+\brch(\ol{\frakl}))X\theta_\brch^{(\frakp)}\text{ if $\ell\ndivides pN$ and $\ell\cO_K=\frakl\ol{\frakl}$ is split},\\
 U_p\sH_1&=\brch(\pbar)\sH_1+\brch(\pbar)((1+u_2)\eta_\frakp(\Frob_{\pbar})-b_1)\cdot X\theta_\brch^{(\frakp)} \\
  U_pX\theta_\brch^{(\frakp)}&=\brch(\pbar)X\theta_\brch^{(\frakp)}.
  \end{align*}
This yields a homomorphism $\lam_{\sH}:\bfT\to {\rm U}\subset \Mat_2(\Cp)$, where ${\rm U}=\stt{\pMX{a}{b}{0}{a} \, \middle| \, a,b\in\cW}$. It is clear that $\lam_{\sH}$ factors through $\bfT^\perp$, and with the identification $\Lam_K/(X^2)\isoto {\rm U},\,X\mapsto \pMX{0}{1}{0}{0}$, we obtain the homomorphism $\lam_{\sH}:\bfT^\perp\to \Lam_P/(X^2)$ with
 \begin{align*}\lam_{\sH}(T_\ell)&=\brch\Psi_2(\Frob_\frakl)+\brch\Psi_2(\Frob_{\ol{\frakl}})\text{ if $\ell\ndivides pN$ and $\ell\cO_K=\frakl\ol{\frakl}$ is split},\\ 
\brch(\pbar)^{-1}\lam_{\sH}(U_p)&=1+ X\left((1+u_2)\eta_\frakp(\Frob_\pbar)-b_1\right).\end{align*}
It follows from \thmref{T:Ribet} that \[\loc_{\pbar}(\kappa)=(1+u_2)\cdot \loc_{\pbar}(\eta_\frakp)-b_1\cdot \Ord_p-\loc_{\pbar}(\psi_2)=\loc_{\pbar}(\eta_\frakp)-\loc_{\pbar}(\eta_\pbar)-b_1\cdot \Ord_p.\]
\subsubsection*{Case (iii): $\Ord_P(A)>\Ord_P(B)>0$} Let $n=\Ord_P(A/B)$ and 
\[u_3=\frac{A}{BX^n}\Big\vert_{X=0}\in F^\x.\]
 Let \[\sH_2:=(-A)\sH\pmod{X^{n+2}}=\bftheta_\brch+\frac{A}{B}\cdot \bftheta_{\brch^c}-A\cdot e_{\Ord}(\theta_\brch^\circ \cG_C) \pmod{X^{n+2}}.\]
 Then $\bfa(1,\sH_2)\con 1\pmod{X}$. Define the additive homomorphism $\psi_3:G_K\to F$ by 
 \[\psi_3=u_3(\eta_\pbar-\chi^{-1}\eta_{\frakp}),\]
and define
 the character $\Psi_3:G_{K,S}\to \Lam_P/(X^{n+2})$ by 
 \begin{align*}\Psi_3
 &=\Psi^{\rm univ}+\psi_3 X^{n+1}\pmod{X^{n+2}}.\end{align*}
Using the equations \eqref{E:43}, \eqref{E:44},
 \begin{align*}\frac{A}{B}\sH_2X&\con\frac{A}{B}\bftheta_\brch X\con\frac{A}{B}\theta_\brch^{(\frakp)}X\con \frac{A}{B}\bftheta_
 {\brch^c} X\pmod{X^{n+2}}; \\
 \left(1-\frac{A}{B}\right)\sH_2X&=\bftheta_\brch X\pmod{X^{n+2}},\end{align*}
we can verify that $\sH_2$ is an eigenform modulo $X^{n+2}$ and 
 \begin{align*}T_\ell \sH_2&=(\brch\Psi_3(\Frob_\frakl)+\brch\Psi_3(\Frob_{\ol{\frakl}}))\sH_2\text{ if $\ell\ndivides pN$ and $\ell\cO_K=\frakl\ol{\frakl}$ is split in $K$},\\
U_p\sH_2&=\brch(\pbar)\left(\Psi^{\rm univ}(\Frob_\pbar)-b_1u_3X^{n+1}\right)\sH_2\pmod{X^{n+2}}.\end{align*}
Likewise we obtain a homomorphism $\lam_\sH:\bfT^\perp\to \Lam_P/(X^{n+2})$ defined by $\lam_\sH(t)=\bfa(1,t\cdot \sH_2)/\bfa(1,\sH_2)$ with 
 \begin{align*}\lam_{\sH}(T_\ell)&=\brch\Psi_3(\Frob_\frakl)+\brch\Psi_3(\Frob_{\ol{\frakl}})\text{ if $\ell\cO_K=\frakl\ol{\frakl}$ is split},\\
\brch(\pbar)^{-1}\lam_{\sH}(U_p)&=\Psi^{\rm univ}(\Frob_\pbar)-b_1u_3X^{n+1}.\end{align*}
It follows from \thmref{T:Ribet} that 
\[\loc_\pbar(\kappa)=-b_1u_3\cdot \Ord_p-\loc_{\pbar}(\psi_3)=u_3(\loc_{\pbar}(\eta_\frakp)-\loc_{\pbar}(\eta_\pbar)-b_1\cdot \Ord_p).\]
In each cases, we see immediately that $\loc_\pbar(\kappa)$ is a multiple of the function in \eqref{E:48}, and the theorem follows.
\end{proof}
\subsection{Proof of  \thmref{T:main}}\label{SS:43}
We are ready to prove \thmref{T:main}. By \remref{R:34}, \[L_p'(0,\chi)=\frac{L_p(s,\chi)}{s}\Big\vert_{s=0}=\cL'_p(s,s,\chi)|_{s=0}.\] 
By \thmref{T:43} and \corref{C:3RS}, we find that the cyclotomic derivative $L_p'(0,\chi)$ equals
\begin{align*}
\cL'_p(s,s,\chi)|_{s=0}
&=2\cL_p'(s,0,\chi)|_{s=0}-\cL'_p(s,-s,\chi)|_{s=0}\\
&=2\sL(\bfone)\cdot \cL_\frakp^*(0,\chi)-\cL_p'(s,-s,\chi)|_{s=0}\\
&=\cL_\frakp^*(0,\chi)\cdot \sL(\chi).
\end{align*}
Now \thmref{T:main} follows from \eqref{E:improved2}.

\def\RomaSuji#1{\setcounter{foo}{#1}\Roman{foo}}%

\def\rs#1{\expandafter{\romannumeral#1}}%
\def\RS#1{\uppercase\rs{#1}}%

\newcommand{\Liminj}[1]{\raisebox{-1.98ex}%
{$\stackrel{\normalsize\mbox{$\varinjlim$}}{\scriptstyle #1}$}}
\newcommand{\Limproj}[1]{\raisebox{-2.00ex}%
{$\stackrel{\normalsize\mbox{$\varprojlim$}}{\scriptstyle #1}$}}

\def\labelenumi{\textup{(\theenumi)}}

\renewcommand{\thefootnote}{\fnsymbol{footnote}}
\section{Comparison of $\mathscr{L}$-invariants}\label{S:comparison}
\subsection{Benois' $\mathscr{L}$-invariant}
Here we briefly recall the definition of $\mathscr{L}$-invariant by Benois \cite{B2, B3, BH}.
Let $p$ be an odd prime.
Let $\varepsilon =(\zeta_{p^n})_{n\geq 0}$ be primitive $p^n$-th roots of unity such that $\zeta_{p^{n+1}}^p=\zeta_{p^n}$
for any $n\geq 0$.
We put $K_n=\mathbb{Q}_p(\zeta_{p^n})$ and $K_\infty = \bigcup_{n\geq 0}K_n$.
Denote $\Gamma =\operatorname{Gal}(K_{\infty}\slash \mathbb{Q}_p)$ and decompose $\Gamma =\Delta \times \Gamma_1$,
where $\Gamma_1 = \operatorname{Gal}(K_\infty \slash K_1)$.
Let $\chi_{\mathrm{cyc}}:\Gamma \to \mathbb{Z}_p^\times$ be the cyclotomic character.
Let $E\slash \mathbb{Q}_p$ be a finite extension. For $r\in [ 0,1)$, we set
$$
\mathscr{R}_E^{(r)}=\left\{ f(x)=\sum_{n\in \mathbb{Z}} a_n X^n \,  \middle| \, a_n \in E, \, f(X) \, \textup{converges on } \{ X \in \mathbb{C}_p  \mid r\leq |X|_p <1 \} \right\} .
$$
Then the Robba ring with coefficients in $E$ is defined by
$
\mathscr{R}_E=\bigcup_{0\leq r < 1} \mathscr{R}_E^{(r)}
$.
The Robba ring $\mathscr{R}_E$ has actions of $\Gamma$ and a Frobenius operator $\varphi$.

For a $(\varphi , \Gamma)$-module $\mathbb{D}$ over the Robba ring $\mathscr{R}_E$, we put
$\mathscr{D}_{\mathrm{cris}}(\mathbb{D})=(\mathbb{D}[1/t])^\Gamma$, where $t=\sum_{n=1}^\infty \frac{X^n}{n}$.
For each $p$-adic representation $V$ of $G_{\mathbb{Q}_p}=\operatorname{Gal}(\overline{\mathbb{Q}}_p\slash \mathbb{Q}_p)$, we can
associate a $(\varphi, \Gamma)$-module $\mathbb{D}_{\mathrm{rig}}^{\dagger}(V)$.
Fix a generator $\gamma_1 \in \Gamma_1$. For any $(\varphi , \Gamma)$-module $\mathbb{D}$, let $\mathit{H}^i(\mathbb{D})$ be the
cohomology of Fontaine-Herr complex
$$
C_{\varphi, \gamma_1}: \mathbb{D}^\Delta \overset{d_0}{\longrightarrow} \mathbb{D}^\Delta  \oplus \mathbb{D}^\Delta  \overset{d_1}{\longrightarrow} \mathbb{D}^\Delta ,
$$
where $d_0(x)=((\varphi -1)x, (\gamma_1 -1)x)$ and $d_1(y,z)=(\gamma_1 -1)y-(\varphi -1)z$.
Let $\mathbb{D}^*(\chi_{\mathrm{cyc}})=\operatorname{Hom}_{\mathscr{R}_E}(\mathbb{D}, \mathscr{R}_E(\chi_{\mathrm{cyc}}))$ be the Tate dual.
For a $(\varphi , \Gamma)$-module $\mathbb{D}$, define
$$\mathit{H}^1_f(\mathbb{D})=\left\{ \alpha \in \mathit{H}^1(\mathbb{D}) \, \middle| \, D_{\alpha} \textup{ is crystalline}\right\} ,$$
where $D_\alpha$ is the extension class associated to $\alpha$.

From now on, we consider the global situation. Fix a finite set of primes $S$ containing $p$ and denote by $\mathbb{Q}_S \slash \mathbb{Q}$ the
maximal Galois extension of $\mathbb{Q}$ unramified outside $S\cup \{ \infty \}$.
We set $G_{\mathbb{Q},S}=\operatorname{Gal}(\mathbb{Q}_S\slash \mathbb{Q})$.
Let $V$ be a $p$-adic representation of $G_{\mathbb{Q}}=\operatorname{Gal}(\overline{\mathbb{Q}}\slash \mathbb{Q})$
unramified outside $S$ with coefficient in a $p$-adic field $E$.
Let $\mathit{H}^1_f(\mathbb{Q},V)$ be the Bloch-Kato Selmer group defined by
$$
\mathit{H}^1_f(\mathbb{Q},V)=
\operatorname{Ker}\left[ \mathit{H}^1(G_{\mathbb{Q},S},V)\to \bigoplus_{v\in S} \frac{\mathit{H}^1(\mathbb{Q}_v,V)}{\mathit{H}_f^1(\mathbb{Q}_v,V)} \right] .
$$
We also denote the relaxed Selmer group by
$$
\mathit{H}^1_{f, \{ p \}}(\mathbb{Q},V)=
\operatorname{Ker}\left[ \mathit{H}^1(G_{\mathbb{Q},S},V)\to \bigoplus_{v\in S\setminus \{ p \}} \frac{\mathit{H}^1(\mathbb{Q}_v,V)}{\mathit{H}_f^1(\mathbb{Q}_v,V)} \right] .
$$

We assume the following conditions:
\begin{itemize}
\item C1) $\mathit{H}^0(G_{\mathbb{Q},S},V)=\mathit{H}^0(G_{\mathbb{Q},S},V^*(1))=0$.
\item C2) $V$ is crystalline at $p$ and $\mathit{D}_{\mathrm{cris}}(V)^{\varphi =1}=0$.
\item C3) The action of $\varphi$ is semisimple on $\mathit{D}_{\mathrm{cris}}(V)$ at $p^{-1}$.
\item C4) $\mathit{H}^1_f(\mathbb{Q},V^*(1))=0$.
\item C5) $\mathrm{loc}_p:\mathit{H}^1_f(\mathbb{Q},V)\to \mathit{H}^1_f(\mathbb{Q}_p,V)$ is injective.
\end{itemize}

\begin{defn}A $\varphi$-submodule $D$ of $\mathit{D}_{\mathrm{cris}}(V)$ is regular if
$D\cap \operatorname{Fil}^0\mathit{D}_{\mathrm{cris}}(V)=0$ and
$r_{V,D}:\mathit{H}^1_f(\mathbb{Q},V) \to \mathit{D}_{\mathrm{cris}}(V)\slash (\operatorname{Fil}^0\mathit{D}_{\mathrm{cris}}(V)+D)$
is an isomorphism, where $r_{V,D}$ is the map induced by
$r_V=\log_V\circ \mathrm{loc}_p:\mathit{H}^1_f(\mathbb{Q},V) \to \mathit{D}_{\mathrm{cris}}(V)\slash \operatorname{Fil}^0\mathit{D}_{\mathrm{cris}}(V)$
and $\log_V$ is the Bloch-Kato logarithm.
\end{defn}
Let $D\subset \mathit{D}_{\mathrm{cris}}(V)$ be a regular submodule. Then we can decompose $D_0=D$ into $D=D_{-1}\oplus D^{\varphi =p^{-1}}$
with $D_{-1}^{\varphi =p^{-1}}=0$.
Let $\mathit{F}_0\mathit{D}_{\mathrm{rig}}^{\dagger}(V)$ and $\mathit{F}_{-1}\mathit{D}_{\mathrm{rig}}^\dagger (V)$ be the $(\varphi, \Gamma)$-modules associated
to $D_0$ and $D_{-1}$ by Berger's theory.
We set $W=\mathrm{gr}_0 \mathit{D}_{\mathrm{rig}}^{\dagger}(V)$. Assume that all the Hodge-Tate weights are non-negative.
Then
$$
i_W:\mathscr{D}_{\mathrm{cris}}(W)\oplus \mathscr{D}_{\mathrm{cris}}(W) \to \mathit{H}^1(W)
$$
defined by $(x,y)\mapsto \mathrm{cl}(-x, y\log \chi_{\mathrm{cyc}})$ is an isomorphism (\cite[Proposition 1.5.9]{B2}).
Let $i_{W,f}$ and $i_{W,c}$ denote the restriction of $i_W$ on the first and second direct summand respectively.
Then we have $\operatorname{Im}(i_{W,f})=\mathit{H}^1_f(W)$ and a decomposition $\mathit{H}^1(W)=\mathit{H}^1_f(W)\oplus \mathit{H}^1_c(W)$,
where $\mathit{H}^1_c(W)=\operatorname{Im}(i_{W,c})$.

For the dual module $W^*(\chi_{\mathrm{cyc}})$, let
$$
i_{W^*(\chi_{\mathrm{cyc}})}:
\mathscr{D}_{\mathrm{cris}}(W^*(\chi_{\mathrm{cyc}}))\oplus \mathscr{D}_{\mathrm{cris}}(W^*(\chi_{\mathrm{cyc}}))
\to \mathit{H}^1(W^*(\chi_{\mathrm{cyc}}))
$$
be the unique linear map such that
$i_{W^*(\chi_{\mathrm{cyc}})}(\alpha , \beta) \cup i_W(x,y)=[\beta ,x]_W -[\alpha ,y]_W$,
where $[ \, , \, ]_W: \mathscr{D}_{\mathrm{cris}}(W^*(\chi_{\mathrm{cyc}}))\times \mathscr{D}_{\mathrm{cris}}(W)\to E$
denotes the canonical pairing induced by $W^*(\chi_{\mathrm{cyc}})\times W \to \mathscr{R}_E(\chi_{\mathrm{cyc}})$.
Similarly, we can define $i_{W^*(\chi_{\mathrm{cyc}}),f}$, $i_{W^*(\chi_{\mathrm{cyc}}),c}$
and $\mathit{H}^1_c(W^*(\chi_{\mathrm{cyc}}))$ using the map $i_{W^*(\chi_{\mathrm{cyc}})}$.

Let
$$
\kappa_D :\mathit{H}^1_{f,\{ p\}}(\mathbb{Q},V)\to \frac{\mathit{H}^1(\mathbb{Q}_p,V)}{\mathit{H}^1_f(\mathit{F}_0\mathit{D}_{\mathrm{rig}}^{\dagger}(V))}
$$
be the composition of the map $\mathrm{loc}_p :\mathit{H}^1_{f,\{ p\}}(\mathbb{Q},V) \to \mathit{H}^1(\mathbb{Q}_p,V)$
and the canonical projection.
Then $\kappa_D$ is an isomorphism (\cite[Lemma 3.1.4]{B3}).
We denote 
$$
\mathit{H}^1(V,D)=
\kappa_D^{-1}(\mathit{H}^1(\mathit{F}_0\mathit{D}_{\mathrm{rig}}^{\dagger}(V))\slash \mathit{H}^1_f(\mathit{F}_0\mathit{D}_{\mathrm{rig}}^{\dagger}(V))).
$$
Then the composition of the map $\mathit{H}^1(V,D)\to \mathit{H}^1(\mathit{F}_0\mathit{D}_{\mathrm{rig}}^{\dagger}(V))\to \mathit{H}^1(W)$
induces an isomorphism $\mathit{H}^1(V,D)\simeq \mathit{H}^1(W)\slash \mathit{H}^1_f(W)$.
We consider the following diagram:
\[\xymatrix{
\mathscr{D}_{\mathrm{cris}}(W)\ar[r]^{i_{W,f}}  & \mathit{H}^1_f(W)  \\
\mathit{H}^1(V,D) \ar[u]^{\rho_{W,f}}  \ar[r]  \ar[d]_{\rho_{W,c}}& \mathit{H}^1(W)\ar[u]_{p_{W,f}} \ar[d]^{p_{W,c}} \\
\mathscr{D}_{\mathrm{cris}}(W)\ar[r]^{i_{W,c}}  & \mathit{H}^1_c(W), \\
}\]
where $p_{W,f}$ and $p_{W,c}$ are the canonical projections, and $\rho_{W,f}$ and $\rho_{W,c}$ are defined
as the unique maps making this diagram commute.
Note that $\rho_{W,c}$ is an isomorphism.

Now we define the $\mathscr{L}$-invariant associated to $V$ and $D$ by
$$\mathscr{L}(V,D)=\operatorname{det}\left( \rho_{W,f}\circ \rho_{W,c}^{-1} \,  \middle| \, \mathscr{D}_{\mathrm{cris}}(W) \right) .$$

\begin{remark}
In \cite{B2}, the choice of the sign of the $\mathscr{L}$-invariant is slightly different from \cite{B3, BH}.
Here we follow the definition given in \cite{B3, BH}.
\end{remark}

Next we consider the dual construction of the $\mathscr{L}$-invariant.
Let $D$ be a regular submodule of $\mathit{D}_{\mathrm{cris}}(V)$ and put
$$
D^\perp =D^\perp_0=\operatorname{Hom}_E(\mathit{D}_{\mathrm{cris}}(V)\slash D , \mathit{D}_{\mathrm{cris}}(E(1)))
$$
and
$$
D^\perp_1 =\operatorname{Hom}_E(\mathit{D}_{\mathrm{cris}}(V)\slash D_{-1} , \mathit{D}_{\mathrm{cris}}(E(1))).
$$
We denote by $\mathit{F}_0\mathit{D}_{\mathrm{rig}}^{\dagger}(V^*(1))$ (resp. $\mathit{F}_1\mathit{D}_{\mathrm{rig}}^{\dagger}(V^*(1))$)
the $(\varphi,\Gamma )$-submodule of $\mathit{D}_{\mathrm{rig}}^{\dagger}(V^*(1))$ associated to $D^\perp_0$ (resp. $D^\perp_1$).
Then we have a short exact sequence
$$
0\to \mathit{F}_1\mathit{D}_{\mathrm{rig}}^{\dagger}(V^*(1)) \to \mathit{F}_0\mathit{D}_{\mathrm{rig}}^{\dagger}(V^*(1)) \to W^*(\chi_{\mathrm{cyc}})\to 0.
$$
Let
$$
\kappa_{D^\perp}:
\mathit{H}^1_f(\mathbb{Q},V^*(1))\to \frac{\mathit{H}^1(\mathbb{Q}_p,V^*(1))}{\mathit{H}^1_f(\mathbb{Q}_p,V^*(1))+\mathit{H}^1(\mathit{F}_0\mathit{D}_{\mathrm{rig}}^{\dagger}(V^*(1)))}
$$
be the map obtained by the composition of $\mathrm{loc}_p$ with the canonical projection.
We set
$$
\mathit{H}^1(V^*(1),D^\perp)=
\kappa_{D^\perp}^{-1}
\left(
\mathit{H}^1(\mathit{F}_1\mathit{D}_{\mathrm{rig}}^{\dagger}(V^*(1)))\slash (\mathit{H}^1_f(\mathbb{Q}_p,V^*(1))+ \mathit{H}^1(\mathit{F}_0\mathit{D}_{\mathrm{rig}}^{\dagger}(V^*(1))))
\right).
$$
Then the composition of the maps
$$
\mathit{H}^1(V^*(1),D^\perp) \to \mathit{H}^1(\mathit{F}_1\mathit{D}_{\mathrm{rig}}^{\dagger}(V^*(1))) \to \mathit{H}^1(W^*(\chi_{\mathrm{cyc}}))
$$
induces an isomorphism $\mathit{H}^1(V^*(1),D^\perp)\simeq \mathit{H}^1(W^*(\chi_{\mathrm{cyc}}))\slash \mathit{H}^1_f(W^*(\chi_{\mathrm{cyc}}))$.
We consider the following diagram:
\[\xymatrix@C=50pt{
\mathscr{D}_{\mathrm{cris}}(W^*(\chi_{\mathrm{cyc}}))\ar[r]^{i_{W^*(\chi_{\mathrm{cyc}}),f}}  & \mathit{H}^1_f(W^*(\chi_{\mathrm{cyc}}))  \\
\mathit{H}^1(V^*(1),D^\perp) \ar[u]^{\rho_{W^*(\chi_{\mathrm{cyc}}),f}}  \ar[r]  \ar[d]_{\rho_{W^*(\chi_{\mathrm{cyc}}),c}}& \mathit{H}^1(W^*(\chi_{\mathrm{cyc}}))\ar[u]_{p_{W^*(\chi_{\mathrm{cyc}}),f}} \ar[d]^{p_{W^*(\chi_{\mathrm{cyc}}),c}} \\
\mathscr{D}_{\mathrm{cris}}(W^*(\chi_{\mathrm{cyc}}))\ar[r]^{i_{W^*(\chi_{\mathrm{cyc}}),c}}  & \mathit{H}^1_c(W^*(\chi_{\mathrm{cyc}})), \\
}\]
where $p_{W^*(\chi_{\mathrm{cyc}}),f}$ and $p_{W^*(\chi_{\mathrm{cyc}}),c}$ are the canonical projections,
and $\rho_{W^*(\chi_{\mathrm{cyc}}),f}$ and $\rho_{W^*(\chi_{\mathrm{cyc}}),c}$ are defined
as the unique maps making this diagram commute.
Note that $\rho_{W^*(\chi_{\mathrm{cyc}}),c}$ is an isomorphism.

We define the $\mathscr{L}$-invariant associated to $V^*(1)$ and $D^\perp$ by
$$\mathscr{L}(V^*(1),D^\perp)=
(-1)^e\operatorname{det}
\left( \rho_{W^*(\chi_{\mathrm{cyc}}),f}\circ \rho_{W^*(\chi_{\mathrm{cyc}}),c}^{-1} \,  \middle| \, \mathscr{D}_{\mathrm{cris}}(W^*(\chi_{\mathrm{cyc}})) \right) ,$$
where $e=\operatorname{dim}_E\mathscr{D}_{\mathrm{cris}}(W^*(\chi_{\mathrm{cyc}}))$.
\begin{prop}\label{dual-formula}
$\mathscr{L}(V^*(1),D^\perp)=(-1)^e
\mathscr{L}(V,D).$
\end{prop}
\begin{proof}
See \cite[Proposition 2.2.7]{B2} and \cite[Proposition 2.3.8]{BH}.
\end{proof}

Using this $\mathscr{L}$-invariant, Benois formulated the exceptional zero conjecture for general crystalline case including non-critical range.

\subsection{Comparison of $\mathscr{L}$-invariants}
Let $K$ be an imaginary quadratic field and $p$ a prime such that $p\mathcal{O}_K=\mathfrak{p}\overline{\mathfrak{p}}$.
Let $\chi: \operatorname{Gal}(H\slash K)\to \overline{\mathbb{Q}}_p^\times$ be a non-trivial ring class character.
Let $E$ be a $p$-adic field containing all of the values of $\chi$.
Assume that $\chi$ is unramified at places above $p$ and $\chi (\overline{\mathfrak{p}})=1$.
Now we consider the case $V=(\operatorname{Ind}^{\mathbb{Q}}_K \chi)^*(\boldsymbol{\varepsilon}_{\rm{cyc}})$.
In this case, we have $V^*(1)=\operatorname{Ind}^{\mathbb{Q}}_K \chi$ and
it is known that
$\mathit{H}^1_f(\mathbb{Q},V)=\mathit{H}^1_f(K,\chi^{-1}(1))=(\mathcal{O}_H^\times \otimes E)[\chi]$,
$\mathit{H}^1_{f, \{ p  \}}(\mathbb{Q},V)=
\mathit{H}^1_{\{ \mathfrak{p}, \overline{\mathfrak{p}}\} }(K,\chi^{-1}(1))=(\mathcal{O}_{H}[1\slash p]^\times \otimes E)[\chi]$
and $\mathit{H}^1_f(\mathbb{Q},V^*(1))=\mathit{H}^1_f(K,\chi)=0$.
For $V=(\operatorname{Ind}^{\mathbb{Q}}_K \chi)^*(1)$, it is easy to see that $V$ satisfies the conditions C1) -- C5).

Denote
$$V^+=\{ v \in V|_{G_K} \mid \sigma (v) =\chi^{-1}(\sigma) \boldsymbol{\varepsilon}_{\rm{cyc}}(\sigma)v \textup{ for all }\sigma \in G_K\}$$
and
$$V^-=\{ v \in V|_{G_K} \mid \sigma (v) =\chi^{-1}(c\sigma c) \boldsymbol{\varepsilon}_{\rm{cyc}}(\sigma)v \textup{ for all }\sigma \in G_K\}.$$
Since $\chi \neq {\chi}^c$, we have a canonical decomposition $V|_{G_K} =V^+\oplus V^-$.
Put
$V_{\mathfrak{p}}=V^+|_{G_{K_{\mathfrak{p}}}}$ and $V_{\overline{\mathfrak{p}}}=V^-|_{G_{K_{\overline{\mathfrak{p}}}}}$.
Then the natural map $\iota : V|_{\mathbb{Q}_p}\to V_{\mathfrak{p}} \oplus V_{\overline{\mathfrak{p}}}$ becomes an isomorphism.
Hence, $\mathit{H}^1(\mathbb{Q}_p,V)=\mathit{H}^1(\mathbb{Q}_p,(\operatorname{Ind}^{\mathbb{Q}}_K \chi)^*(1))$ can be identified with
$\mathit{H}^1(K_{\mathfrak{p}},\chi^{-1}(1))\oplus \mathit{H}^1(K_{\overline{\mathfrak{p}}},\chi^{-1}(1))$.
\begin{defn}
We choose a regular submodule $D$ of $\mathit{D}_{\mathrm{cris}}(V)$ as 
$D=\mathit{D}_{\mathrm{cris}}(\iota^{-1}(V_{\overline{\mathfrak{p}}}))$.
\end{defn}
Then $\mathit{H}^1(\mathit{F}_0\mathit{D}_{\mathrm{rig}}^\dagger (V))$ is identified with $\mathit{H}^1(K_{\overline{\mathfrak{p}}},\chi^{-1}(1))$
under the isomorphism $\mathit{H}^1(\mathit{D}_{\mathrm{rig}}^\dagger (V))\simeq \mathit{H}^1(\mathbb{Q}_p,V)$.
Here we recall that $\mathit{F}_0\mathit{D}_{\mathrm{rig}}^\dagger (V)$ is the $(\varphi,\Gamma)$-submodule of
$\mathit{D}_{\mathrm{rig}}^\dagger (V)$ associated to $D$.
This property also characterizes the choice of the regular submodule $D$.
Then the modified Euler factor associated to $(V, D)$ is given by
$$
\mathcal{E} (V,D)= \operatorname{det}(1-p^{-1}\varphi^{-1}|D )   \operatorname{det}(1-\varphi | \mathit{D}_{\mathrm{cris}}(V)\slash D ) 
=(1-\chi (\overline{\mathfrak{p}}))(1-\chi^{-1}(\mathfrak{p})p^{-1})=0
$$ 
and 
$$\mathcal{E}^+ (V,D)= \operatorname{det}(1-p^{-1}\varphi^{-1}|D_{-1} )   \operatorname{det}(1-p^{-1}\varphi^{-1} | D^\perp ) 
=(1-\chi (\mathfrak{p})p^{-1}),$$
where $\mathcal{E}^+ (V,D)$ is the modified Euler factor which is used in the formula of the exceptional zero conjecture (\cite[Conjecture 4]{B3}).
Note that $\mathcal{E}^+ (V,D)$ coincides with the Euler factor appeared in Conjecture \ref{conj:main}.
Therefore Conjecture \ref{conj:main} is compatible with the exceptional zero conjecture formulated by Benois. 
\begin{prop}
We have $\mathscr{L}(V,D)=-\mathscr{L}(\chi)$, where $\mathscr{L}(\chi)$ is the $\mathscr{L}$-invariant defined in (\ref{E:Linv}).
\end{prop}
\begin{proof}
In this case, $\mathit{F}_0\mathit{D}_{\mathrm{rig}}^\dagger (V)\simeq \mathscr{R}_E(|x|x)$ and $\mathit{F}_{-1}\mathit{D}_{\mathrm{rig}}^\dagger (V)=0$,
where we write $x$ for the character given by the identity map and $|x|$ for $|x|=p^{v_p(x)}$.
Hence we have $W=\operatorname{gr}_0\mathit{D}_{\mathrm{rig}}^\dagger (V)\simeq \mathscr{R}_E(|x|x)$ and
$\mathit{H}^1(W)\simeq \mathit{H}^1(\mathscr{R}_E(|x|x))\simeq \mathit{H}^1(\mathscr{R}_E(\chi_{\mathrm{cyc}}))\simeq \mathit{H}^1(\mathbb{Q}_p, E(1))$.

Define $\alpha_W=i_{W,f}(1)$ and $\beta_W=i_{W,c}(1)$. Let $\kappa : \mathbb{Q}_p^\times \otimes E \to \mathit{H}^1(\mathbb{Q}_p,E(1))$ be the
Kummer map.
Then we have 
$p_{W,f}(\kappa (u))= \log_p u \cdot \alpha_W$ and $p_{W,c}(\kappa (u))=\operatorname{ord}_p(u) \cdot \beta_W$
for $u \in \mathbb{Q}_p^\times \otimes E$ (see \cite[1.5.6 and 1.5.10]{B2} for details).
Since 
$\mathit{H}^1_{f, \{ p   \}}(\mathbb{Q},V)=(\mathcal{O}_{H}[1\slash p]^\times \otimes E)[\chi]$
and $\mathit{H}^1(\mathit{F}_0\mathit{D}_{\mathrm{rig}}^\dagger (V))=\mathit{H}^1(K_{\overline{\mathfrak{p}}},\chi^*(1))$,
one has
\begin{align*}
\mathit{H}^1(V,D)&=
\kappa_D^{-1}(\mathit{H}^1(\mathit{F}_0\mathit{D}_{\mathrm{rig}}^{\dagger}(V))\slash \mathit{H}^1_f(\mathit{F}_0\mathit{D}_{\mathrm{rig}}^{\dagger}(V)))\\
&=\operatorname{Ker}\left[  \mathit{H}^1_{f, \{ p   \}}(\mathbb{Q},V) \to  \mathit{H}^1(K_{\mathfrak{p}},\chi^*(1)) \right] \\
&=\operatorname{Ker}\left[ (\mathcal{O}_{H}[1\slash p]^\times \otimes E)[\chi] \to H_{\mathfrak{p}}^\times \otimes E \right] .
\end{align*}
In this case, $\mathit{H}^1(V,D)$ is an one-dimensional $E$-vector space.

Let $\operatorname{ord}_{\overline{\mathfrak{p}}}$ and $\log_{\overline{\mathfrak{p}}}$ be the elements in
$\mathit{H}^1(K_{\overline{\mathfrak{p}}},\mathbb{Q}_p)=\operatorname{Hom}(G_{K_{\overline{\mathfrak{p}}}},\mathbb{Q}_p)$
corresponding to $\operatorname{ord}_p$ and $\log_p$ under the identification 
$\operatorname{Hom}(G_{K_{\overline{\mathfrak{p}}}},\mathbb{Q}_p)=\operatorname{Hom}(G_{\mathbb{Q}_p},\mathbb{Q}_p)$.
They can be viewed as maps $\operatorname{ord}_{\overline{\mathfrak{p}}}, \log_{\overline{\mathfrak{p}}} : K_{\overline{\mathfrak{p}}}^\times \to \mathbb{Q}_p$
via the geometrically normalized reciprocity law map $\operatorname{rec}_{\overline{\mathfrak{p}}}:K_{\overline{\mathfrak{p}}}^\times \to G_{K_{\overline{\mathfrak{p}}}}$.

We fix a non-zero element $u$ in the one-dimensional $E$-vector space
$
\mathit{H}^1(V,D)
$.
Then we have
$$\mathscr{L}(V,D)=\frac{\log_{\overline{\mathfrak{p}}}(u)}{\operatorname{ord}_{\overline{\mathfrak{p}}}(u)}$$
by the definition.
This shows $\mathscr{L}(V,D)=-\mathscr{L}(\chi)$.
\end{proof}

Next we compute the dual construction.
In this case, it is easy to see
$$\mathit{H}^1(\mathit{F}_1\mathit{D}_{\mathrm{rig}}^\dagger (V^*(1)))
=\mathit{H}^1(\mathit{D}_{\mathrm{rig}}^\dagger (V^*(1)))
=\mathit{H}^1(\mathbb{Q}_p,V^*(1))
=\mathit{H}^1(K_{\mathfrak{p}},\chi)\oplus \mathit{H}^1(K_{\overline{\mathfrak{p}}},\chi),
$$
$$\mathit{H}^1(\mathit{F}_0\mathit{D}_{\mathrm{rig}}^\dagger (V^*(1)))=\mathit{H}^1(K_{\mathfrak{p}},\chi)
$$
and
$$
\mathit{H}^1_f(\mathbb{Q}_p,V^*(1))=\mathit{H}^1_f(K_{\mathfrak{p}},\chi)\oplus \mathit{H}^1_f(K_{\overline{\mathfrak{p}}},\chi)=0.
$$
Hence we have 
$$
\mathit{H}^1(V^*(1),D^\perp)
=\mathit{H}^1_{f, \{ p \} }(\mathbb{Q},V^*(1))=\mathit{H}^1_{f, \{ \mathfrak{p}, \overline{\mathfrak{p}} \} }(K,\chi) 
=\mathit{H}^1(K,\chi),
$$
which is an one-dimensional $E$-vector space.
Moreover 
$$
\mathit{H}^1(W^*(\chi_{\mathrm{cyc}}))\simeq 
\mathit{H}^1(\mathit{F}_1\mathit{D}_{\mathrm{rig}}^{\dagger}(V^*(1)))\slash \mathit{H}^1(\mathit{F}_0\mathit{D}_{\mathrm{rig}}^{\dagger}(V^*(1)))
\simeq \mathit{H}^1(K_{\overline{\mathfrak{p}}},\chi)\simeq \mathit{H}^1(\mathbb{Q}_p,E).
$$
Define
$\alpha_{W^*(\chi_{\mathrm{cyc}})}=i_{W^*(\chi_{\mathrm{cyc}}),f}(1)$ and $\beta_{W^*(\chi_{\mathrm{cyc}})}=i_{W^*(\chi_{\mathrm{cyc}}),c}(1)$.
Under the identification
$$\mathit{H}^1(W^*(\chi_{\mathrm{cyc}}))\simeq \mathit{H}^1(\mathscr{R}_E)\simeq \mathit{H}^1(\mathbb{Q}_p,E),$$
one has $\alpha_{W^*(\chi_{\mathrm{cyc}})}=-\operatorname{ord}_p$ and $\beta_{W^*(\chi_{\mathrm{cyc}})}=\log_p$ (see \cite[1.5.6 and 1.5.10]{B2}).
Note that our normalization of the reciprocity law map is different from Benois \cite{B2,B3,BH}.
More precisely, we have $\operatorname{ord}_p (\mathrm{Fr}_p)=1$,
where $\mathrm{Fr}_p$ is the geometric Frobenius. This gives the difference of the sign with Benois' description.

Fix a non-zero element $\eta \in \mathit{H}^1(V^*(1),D^\perp)=\mathit{H}^1(K,\chi)$.
Then we can write 
$$\kappa_{D^\perp}(\eta)=x\cdot \operatorname{ord}_{\overline{\mathfrak{p}}}+y\cdot \log_{\overline{\mathfrak{p}}}
=(-x)\cdot (-\operatorname{ord}_{\overline{\mathfrak{p}}})+y\cdot \log_{\overline{\mathfrak{p}}}$$
in $\mathit{H}^1(W^*(\chi_{\mathrm{cyc}}))\simeq\mathit{H}^1(K_{\overline{\mathfrak{p}}},\chi)$ and we have
$\displaystyle \mathscr{L}(V^*(1),D^\perp)=(-1)^e \left( -\frac{x}{y}\right)$,
where $e=\operatorname{dim}_E D^{\varphi=p^{-1}}=1$.
By Proposition \ref{dual-formula}, we get $\mathscr{L}(V^*(1),D^\perp)=-\mathscr{L}(V,D)=\mathscr{L}(\chi)$ again.
Therefore this gives an alternative proof of Lemma \ref{L:41}.

\section*{Acknowledgement}This work has its root in the authors' participation in the program \emph{Advancing Strategic International Networks to Accelerate
the Circulation of Talented Researchers} during 2015--2017. The authors would like to thank Shinichi Kobayashi and Nobuo Tsuzuki for their support and hospitality during the period of this program. The authors are grateful to the referees for the suggestions and comments on the improvement of the manuscript.

\bibliographystyle{amsalpha}
\bibliography{mybib}
\end{document}